\documentclass[11pt]{article}

\usepackage[margin=1.3in]{geometry}
\usepackage{mathtools}
\usepackage[dvips]{graphics}
\usepackage{epsfig}
\usepackage{multirow}
\usepackage{array}
\usepackage{extarrows}
\usepackage{graphicx}
\mathtoolsset{showonlyrefs}
\usepackage{algorithm}
\usepackage{algorithmic}
\usepackage[symbol]{footmisc}

\newcommand{\INPUT}{\REQUIRE}
\newcommand{\OUTPUT}{\ENSURE}

\usepackage{comment}
\usepackage[titletoc,title]{appendix}
\numberwithin{equation}{section}
\usepackage{amsthm}
\usepackage{enumerate}
\usepackage{amsmath}

\usepackage{amssymb,amsmath}
\usepackage{latexsym}
\usepackage[usenames]{xcolor}
\usepackage{pgf}
\usepackage[numbers]{natbib}

\usepackage{hyperref}
\hypersetup{
    colorlinks,%
    citecolor=black,%
    filecolor=black,%
    linkcolor=black,%
    urlcolor=black
}
\usepackage{subfig}
\numberwithin{equation}{section}

\newtheorem{theorem}{ \noindent T{\footnotesize HEOREM}}
\newtheorem{prop}{ \noindent P{\footnotesize ROPOSITION}}[section]
\newtheorem{lemma}{ \noindent L{\footnotesize EMMA}}[section]
\newtheorem{coro}{ \noindent C{\footnotesize OROLLARY}}
\newtheorem{assumption}{ \noindent A{\footnotesize SSUMPTION}}
\newtheorem{remark}{ \noindent R{\footnotesize EMARK}}[section]

\newcommand{\E}{\mathbb{E}}

\newcommand{\R}{\mathbb{R}}

\newcommand{\T}{\top}

\begin{document}

\title{A note on the improved sparse Hanson-Wright inequalities}
\author{Guozheng Dai$^\ddagger$, Yiyun He$^*$, Ke Wang$^\dagger$ and Yizhe Zhu$^\S$}

\date{}
\maketitle

\begin{abstract}
We establish sparse Hanson–Wright inequalities for quadratic forms of sparse $\alpha$-sub-exponential random vectors with exponent parameter $\alpha\in(0, 2]$.  In the regime $0< \alpha\le 1$ we derive a refined inequality that is optimal in several canonical models.  These results extend the classical Hanson–Wright bound to the sparse setting.  Illustrative applications include covariance matrix estimation with incomplete observations, low-rank matrix approximation under the maximum norm with sparsified sketches, and concentration inequalities for sparse $\alpha$-sub-exponential random vectors.
\end{abstract}

\footnotetext[3]{Department of Mathematics, Hong Kong University of Science and Technology, Clear Water Bay, Kowloon, Hong Kong, guozhengdai@ust.hk. }
\footnotetext[1]{Department of Mathematics, University of California San Diego, La Jolla, CA 92093, USA, yih@ucsd.edu.}
\footnotetext[2]{Department of Mathematics, Hong Kong University of Science and Technology, Clear Water Bay, Kowloon, Hong Kong, kewang@ust.hk. }
\footnotetext[4]{Department of Mathematics, University of Southern California, 3620 Vermont Avenue
Los Angeles, CA 90089, USA, yizhezhu@usc.edu. }

\bigskip

\section{Introduction}

The Hanson-Wright  inequality quantifies the concentration of a random quadratic form around its mean. Introduced in the 1970s \cite{Hanson_AMS}, it is now a standard device for covariance estimation, random-matrix analysis, empirical-process bounds and high-dimensional regression.

Let $\xi=(\xi_{1}, \cdots, \xi_{n})^{\top}$ be a random vector with independent centered entries and $A=(a_{ij})_{n\times n}$ be a fixed matrix. The Hanson-Wright inequality is to explore the concentration properties of the following quadratic random variables
 $$S_{A}(\xi):=\xi^\top A\xi=\sum_{i,j}a_{ij}\xi_{i}\xi_{j}.$$ 
 A well-known result, proved by Hanson and Wright \cite{Hanson_AMS}, claims that if $\xi_{i}$ are independent centered $2$-sub-exponential (sub-gaussian) random variables such that $\max_{i}\Vert \xi_{i}\Vert_{\Psi_{2}}\le L$, then for $t\ge 0$ (the following version was provided in \cite{Rudelson_ECP})

\begin{align}\label{Eq_Intro_HW_subgaussian}
	\mathbb{P}\big\{ \vert S_{A}(\xi)-\mathbb{E}S_{A}(\xi) \vert\ge t   \big\}\le 2\exp\Big(-c\min\Big\{\frac{t^{2}}{L^{4}\Vert A\Vert_{F}^{2}}, \frac{t}{L^{2}\Vert A\Vert_{l_{2}\to l_{2}}}   \Big\}   \Big),
\end{align}
where $\Vert \cdot\Vert_{F}$ and $\Vert\cdot\Vert_{l_{2}\to l_{2}}$ are the Frobenius norm and the spectral norm of a matrix respectively. Recall that a random variable $\eta_{1}$ is $\alpha$-sub-exponential if satisfying
\begin{align}
	\mathbb{P}\{\vert\eta_{1}\vert\ge t   \}\le 2\exp\big(-\frac{t^{\alpha}}{K^{\alpha}}\big),\quad t\ge 0.\nonumber
\end{align}
Its $\alpha$-sub-exponential norm is defined as 
\begin{align}
	\Vert \eta_{1}\Vert_{\Psi_{\alpha}}:=\inf \big\{ t>0:  \mathbb{E}\exp\big(\frac{\vert \eta_{1}\vert^{\alpha}}{t^{\alpha}}   \big)\le 2   \big\} .\nonumber
\end{align}

There is a lot of work dedicated to extending the classical Hanson-Wright inequality \eqref{Eq_Intro_HW_subgaussian} to more general cases. For example,  assume that $\xi_{1}, \cdots, \xi_{n}$ are independent centered $\alpha$-sub-exponential  variables and satisfy $\max_{i}\Vert \xi_{i}\Vert_{\Psi_{\alpha}}\le  L(\alpha)$. When $1\le \alpha\le 2$, Adamczak and Lata{\l}a \cite{Adamczak_AIHP}  proved for $t\ge 0$ (see also \cite{Adamczak_PTRF})
\begin{align}\label{Eq_intro_logconcave}
	\mathbb{P}\{\vert S_{A}(\xi)-\mathbb{E}S_{A}(\xi) \vert\ge L(\alpha)^{2}t    \}\le 2\exp(-c(\alpha)f_{1}(t)),
\end{align}
where 
\begin{align}\label{Eq_intro_f1}
	f_{1}(t)=\min\Big\{\Big( \frac{t}{\Vert A\Vert_{F}}\Big)^{2}, 
	& \frac{t}{\Vert A\Vert_{l_{2}\to l_{2}} }, \Big( \frac{t}{\Vert A\Vert_{l_{\alpha^{*}}(l_{2})}} \Big)^{\alpha},\nonumber\\
	&\Big( \frac{t}{\Vert A\Vert_{l_{2}\to l_{\alpha^{*}}}} \Big)^{\frac{2\alpha}{2+\alpha}}, \Big(\frac{t}{\Vert A\Vert_{l_{\alpha}\to l_{\alpha^{*}}}}\Big)^{\frac{\alpha}{2}}\Big\}.
\end{align}
One can refer to \eqref{matrix_norm} below for the definitions of the matrix's norms and $\alpha^{*}=\alpha/(\alpha-1)$.
When $0<\alpha\le 1$, Kolesko and Lata{\l}a \cite{Kolesko_SPL} proved for $t\ge 0$ (see also \cite{Gotze_EJP})
\begin{align}\label{Eq_intro_logconvex}
		\mathbb{P}\{\vert S_{A}(\xi)-\mathbb{E}S_{A}(\xi) \vert\ge L(\alpha)^{2}t    \}\le 2\exp(-c(\alpha)f_{2}(t)),
\end{align}
	where 
	\begin{align}\label{Eq_intro_f2}
		f_{2}(t)=\min\Big\{ \Big( \frac{t}{\Vert A\Vert_{F}}\Big)^{2}, 
		\frac{t}{\Vert A\Vert_{l_{2}\to l_{2}} }, \Big( \frac{t}{\Vert A\Vert_{l_2\to l_{\infty}}} \Big)^{\frac{2\alpha}{2+\alpha}}, \Big(\frac{t}{\Vert A\Vert_{\infty}}\Big)^{\frac{\alpha}{2}}\Big\}.
	\end{align}
As we have seen above, the Hanson-Wright inequalities \eqref{Eq_intro_logconcave} and \eqref{Eq_intro_logconvex} are quite intricate.  By evaluating the family of norms used therein,
Sambale \cite{Sambale_Progress_probability} obtained a simplified version. This simplified version is not only easily calculable but also sufficient for many applications (see \cite{Sambale_Progress_probability} for details). In particular, Sambale proved for  $0<\alpha\le 2$ and $t\ge 0$
\begin{align}\label{Eq_Intro_Sambale_HW}
	&\mathbb{P}\big\{ \vert S_{A}(\xi)-\mathbb{E}S_{A}(\xi) \vert\ge t   \big\}\nonumber\\
	\le& 2\exp\Big(-c(\alpha)\min\Big\{\frac{t^{2}}{L(\alpha)^{4}\Vert A\Vert_{F}^{2}}, \Big(\frac{t}{L(\alpha)^{2}\Vert A\Vert_{l_{2}\to l_{2}}}\Big)^{\alpha/2}   \Big\}   \Big).
\end{align}

Motivated by the covariance estimation problem in the matrix variate model, Zhou \cite{Zhou_Bernoulli} showed a sparse Hanson-Wright type inequality for sub-gaussian random variables. In particular, assume that $\{\xi_{i}=\delta_{i}\cdot\zeta_{i}, 1\le i\le n\}$ is a sequence of independent random variables, where $\delta_{i}\sim\textnormal{Bernoulli}(p_{i})$ (that is, taking values $0, 1$ with probability $1-p_{i}$ and $p_{i}$ respectively) and $\zeta_{i}$ is a  centered sub-gaussian variable independent of $\delta_{i}$. Zhou proved for $t\ge 0$
\begin{align}\label{Eq_intro_zhou_HW}
	\mathbb{P}\big\{ \vert S_{A}(\xi)-\mathbb{E}S_{A}(\xi) \vert\ge t   \big\}
	\le 2\exp\Big(-c\min\Big\{\frac{t^{2}}{L^{4}\gamma_{1}(A)}, \frac{t}{L^{2}\Vert A\Vert_{l_{2}\to l_{2}}}   \Big\}   \Big),
\end{align}
where $L=\max_{i}\Vert\zeta_{i}\Vert_{\Psi_{2}}$ and $\gamma_{1}(A)=\sum_{k}a_{kk}^{2}p_{k}+\sum_{i\neq j}a_{ij}^{2}p_{i}p_{j}$.

Recently, He, Wang and Zhu \cite{He_arXiv_HW} established sparse Hanson-Wright type inequalities in a more general setting. Specifically, let $\{ \xi_{i}=\delta_{i}\cdot \zeta_{i}, 1\le i\le n\}$ be a sequence of independent random variables, where $\delta_{i}\sim \textnormal{Bernoulli}(p_{i})$ and $\zeta_{i}$ is a centered $\alpha$-sub-exponential variable independent of $\delta_{i}$, $1\le i\le n$. \cite[Corollary 1]{He_arXiv_HW} showed that for $\alpha>0$ and $t\ge 0$
\begin{align}\label{Eq_Intro_He_Hw}
		&\mathbb{P}\big\{ \vert S_{A}(\xi)-\mathbb{E}S_{A}(\xi) \vert\ge t   \big\}\nonumber\\
	\le& 2\exp\Big(-c(\alpha)\min\Big\{\frac{t^{2}}{L(\alpha)^{4}\gamma_{1}(A)}, \frac{t}{L(\alpha)^{2}\gamma_{2}(A)},  \Big(\frac{t}{L(\alpha)^{2}\Vert A\Vert_{\infty}}\Big)^{\min\{\frac{\alpha}{2}, \frac{1}{2}\}}   \Big\}   \Big),
\end{align}
where $L(\alpha)=\max_{i}\Vert \zeta_{i}\Vert_{\Psi_{\alpha}}$, $\Vert A\Vert_{\infty}=\max_{ij}\vert a_{ij}\vert$, and
\begin{align}
	\gamma_{2}(A)=\max_{i}\big\{\sum_{j:j\neq i}\vert a_{ij}\vert p_{j},  \sum_{j:j\neq i}\vert a_{ji}\vert p_{j}, \vert a_{ii}\vert  \big\}.\nonumber
\end{align}
The work in \cite{He_arXiv_HW} further extends to sparse versions of Bernstein and Bennett type inequalities.

In addition to the works mentioned above, there are numerous other papers that delve into the Hanson-Wright inequality with sparse structure. We refer interested readers to \cite{CK25,GLZ00,He_aop,Park_Stat,schudy2011bernstein,schudy2012concentration,wu2023precise,Zhou_Bernoulli_2024} and the references therein for further exploration. 

\paragraph{Notation}
Given a fixed vector $x=(x_{1},\cdots, x_{n})^\top\in \mathbb{R}^{n}$,  denote by $\Vert x\Vert_{r}=(\sum_i\vert x_{i}\vert^{r})^{1/r}$ the $l_{r}$ norm. Use $\Vert \xi\Vert_{L_{r}}=(\mathbb{E}\vert \xi\vert^{r})^{1/r}$ as the $L_{r}$ norm of a random variable $\xi$.
  Let $A=(a_{ij})$ be an $m\times n$ matrix. We use the following notations of the matrix norms:
\begin{align}
	\Vert A\Vert_{F}=&\sqrt{\sum_{i,j}\vert a_{ij}\vert^{2}}, \quad \Vert A\Vert_{\infty}=\max_{i,j}\vert a_{ij}\vert, \\
    \Vert A\Vert_{l_{r}(l_{2})}&=(\sum_{i\le m}(\sum_{j\le n}\vert a_{ij}\vert^{2})^{r/2})^{1/r}, \quad  \label{matrix_norm}\\
	\Vert A\Vert_{l_{r_{1}}\to l_{r_{2}}}&=\sup \{\vert\sum_{i,j} a_{ij}x_{j}y_{i}\vert: \Vert x\Vert_{r_{1}}\le 1, \Vert y\Vert_{r_{2}^{*}} \le 1  \},\nonumber
\end{align}
where $1\le r_{1}, r_{2}<\infty$ and $ r_{2}^{*}=r_{2}/(r_{2}-1)$. Indeed, $\Vert A\Vert_{l_{2}\to l_{2}}$ is the spectral norm of $A$ and $\Vert A\Vert_{l_{1}\to l_{\infty}}=\Vert A\Vert_{\infty}$. 

Let $\xi_{1}, \cdots, \xi_{n}$ be a sequence of random variables and $f: \mathbb{R}^{n}\to \mathbb{R}$ be a measurable function. Then, $\mathbb{E}_{\xi_{i_{1}}, \cdots, \xi_{i_{s}}}f(\xi_{1}, \cdots, \xi_{n})$ means only taking expectation with respect to random variables $\xi_{i_{1}}, \cdots, \xi_{i_s}$, where $\{i_{1}, \cdots, i_{s}  \}\subset\{1,\cdots, n  \}$. Denote $\xi\sim \mathcal{W}_{s}(\alpha)$ when $\xi$ is  a symmetric Weibull variable with the scale parameter $1$ and the shape parameter $\alpha$. Specifically, $-\log \mathbb{P}\{\vert\xi\vert>x  \}=x^{\alpha}, x\ge 0$.

Unless otherwise stated,  denote by $C, C_{1}, c, c_{1},\cdots$ the universal constants independent of any parameters and the dimension $n$. Besides, let $C(\delta), c(\delta)\cdots $ be the constants depending only on the parameter $\delta$. Their values can change from line to line. For convenience, we say $f\lesssim g$ if $f\le Cg$ for some universal constant $C$. We write $f\lesssim_{\delta} g$ if $f\le C(\delta)g$ for some constant $C(\delta)$. Besides, we say $f\asymp g$ if $f\lesssim g$ and $g\lesssim f$, so does $f\asymp_{\delta} g$.

\paragraph{Organization of the paper} 

The rest of the paper is organized as follows. In the remaining part of Section 1, we will first present our main results. Then, we will draw comparisons between our results and existing ones to highlight the novelties and improvements. To conclude Section 1, we will pose two open questions that warrant further exploration.

In Section 2, we give three applications of our concentration inequalities. The applications include: (i) covariance matrix estimation from partially observed data; (ii) low-rank matrix approximation with sparsification; and (iii) concentration phenomena for sparse $\alpha$-sub-exponential random vectors. 

Section 3 is dedicated to introducing several key lemmas. These lemmas cover aspects such as the relationship between tails and moments, the decoupling inequality, the contraction principle, and concentration inequalities. They serve as fundamental tools for the subsequent theoretical derivations.

We  prove our main results in Section 4 and our applications in Section 5. In the appendix, we offer proofs for Proposition \ref{Lem_appendix_1}, Theorem \ref{Theorem_application}, Corollary \ref{Cor_Section2_quadratic_forms}, and Lemma \ref{Prop_Section3_Bernstein_sparse}.

\subsection{Main results}
This paper concentrates on the sparse Hanson-Wright inequalities in $\alpha$-sub-exponential random variables, $0<\alpha\le 2$. Our first main result in this regime reads as follows: 
\begin{theorem}\label{Theo_main1}
	Assume that $\{\delta_{i}, 1\le i\le n \}$ and $\{\zeta_{i}, 1\le i\le n \}$ are two independent sequences of random variables, where $\delta_{i}\sim\textnormal{Bernoulli}(p_{i})$ are independent and $\zeta_{i}$ are independent centered $\alpha$-sub-exponential random variables.
	Consider a random vector $\xi=(\xi_{1}, \cdots, \xi_{n})$  with $\xi_{i}=\delta_{i}\cdot\zeta_{i}$. Let $A=(a_{ij})_{n\times n}$ be a symmetric fixed matrix. Then for $0<\alpha\le 2$ and $t\ge 0$
	\begin{align}
		&\mathbb{P}\{\vert \xi^\top A\xi-\mathbb{E}\xi^\top A\xi\vert\ge L(\alpha)^{2}t  \}\nonumber\\
		\le& 2\exp\left(-c(\alpha)\min\left\{ \frac{t^{2}}{\sum_{k}a_{kk}^{2}p_{k}+\sum_{i\neq j}a_{ij}^{2}p_{i}p_{j}}, \left(\frac{t}{\Vert A\Vert_{l_{2}\to l_{2}}}\right)^{\alpha/2}  \right\}      \right),\nonumber
	\end{align}
where $L(\alpha)=\max_{i}\Vert \zeta_{i}\Vert_{\Psi_{\alpha}}$ and $c(\alpha)$ is a positive constant depending only on $\alpha$.
\end{theorem}

\begin{remark}[Comparison of Theorem~\ref{Theo_main1} with Hanson-Wright inequalities in \cite{Sambale_Progress_probability,Zhou_Bernoulli}] 
(i). When $p_{1}=\cdots=p_{n}=1$, i.e., $\delta_{i}=1$ for $1\le i\le n$, Theorem \ref{Theo_main1} recovers Sambale's result \cite{Sambale_Progress_probability}, see \eqref{Eq_Intro_Sambale_HW}.  
(ii). When $\alpha=2$, Theorem \ref{Theo_main1} recovers the sparse Hanson-Wright inequality for sparse sub-Gaussian random vectors obtained by Zhou in \cite{Zhou_Bernoulli}, see \eqref{Eq_intro_zhou_HW}. 
\end{remark}

Moreover, when only considering the case $0<\alpha\le 1$, we have the following more refined result.
\begin{theorem}\label{Theo_main2}
	In the same setting of Theorem \ref{Theo_main1}, we have for $0<\alpha\le 1$ and $t\ge 0$
	\begin{align}
		\mathbb{P}\{\vert \xi^\top A\xi-\mathbb{E}\xi^\top A\xi\vert\ge L(\alpha)^{2}t    \}\le 2\exp(-c(\alpha)f(t)),\nonumber
	\end{align}
where  $L(\alpha)=\max_{i}\Vert \zeta_{i}\Vert_{\Psi_{\alpha}}$, $c(\alpha)$ is a positive constant depending only on $\alpha$, and
\begin{align}
	f(t)=\min\Big\{ \frac{t^{2}}{\sum_{k}a_{kk}^{2}p_{k}+\sum_{i\neq j}a_{ij}^{2}p_{i}p_{j}}, 
	& \frac{t}{\Vert (a_{ij}\sqrt{p_{i}p_{j}})_{n\times n}\Vert_{l_{2}\to l_{2}} }, \nonumber\\
	&\Big( \frac{t}{\max_{i}\big( \sum_{j}a_{ij}^{2}p_{j}\big)^{1/2}} \Big)^{\frac{2\alpha}{2+\alpha}}, \Big(\frac{t}{\Vert A\Vert_{\infty}}\Big)^{\frac{\alpha}{2}}\Big\}.\nonumber
\end{align}
\end{theorem}

\begin{remark}[Theorem~\ref{Theo_main2} improves Theorem~\ref{Theo_main1}]
     Theorem \ref{Theo_main1} for the case $0<\alpha\le 1$ can be derived from Theorem \ref{Theo_main2}. Indeed, a direct integration of the tail probability in Theorem \ref{Theo_main2} yields that 
\begin{align}
	\Big\Vert  \xi^\top A\xi-\mathbb{E}\xi^{\top}A\xi \Big\Vert_{L_{r}}\lesssim_{\alpha}&r^{2/\alpha}\Vert A\Vert_{\infty}+r^{1/2+1/\alpha}\max_{i} \Big(\sum_{j}a_{ij}^{2}p_{j} \Big)^{1/2}\nonumber\\
	&+r\Vert (a_{ij}\sqrt{p_{i}p_{j}})\Vert_{l_{2}\to l_{2}}+\sqrt{r}\Big(\sum_{k}a_{kk}^{2}p_{k}+\sum_{i\neq j}a_{ij}^{2}p_{i}p_{j} \Big)^{1/2}\nonumber\\
	\lesssim_{\alpha} &r^{2/\alpha}\Vert A\Vert_{l_{2}\to l_{2}}+\sqrt{r}\Big(\sum_{k}a_{kk}^{2}p_{k}+\sum_{i\neq j}a_{ij}^{2}p_{i}p_{j} \Big)^{1/2},\nonumber
\end{align}
where we use the following fact
\begin{align}
 &\max_{i} \Big(\sum_{j}a_{ij}^{2}p_{j} \Big)^{1/2}\leq \|A\|_{l_2\to l_{\infty}}\leq \|A\|_{l_2\to l_2},\nonumber\\
 &\Vert (a_{ij}\sqrt{p_{i}p_{j}})\Vert_{l_{2}\to l_{2}}\le \Vert A\Vert_{l_{2}\to l_{2}},\quad \|A\|_{\infty}\leq \Vert A\Vert_{l_{2}\to l_{2}}.
\end{align}
Then, one can obtain Theorem \ref{Theo_main1} by Lemma \ref{Lem_Moments_to_tailbound} below.
\end{remark}

\begin{remark}[Comparison of Theorem~\ref{Theo_main2} with Hanson-Wright inequalities  in \cite{He_arXiv_HW,Kolesko_SPL}]

(i).  When $p_{1}=\cdots=p_{2}=1$, Theorem \ref{Theo_main2} recovers the Hanson-Wright inequality \eqref{Eq_intro_logconvex} due to Kolesko and Lata{\l}a in \cite{Kolesko_SPL}.

(ii). Note that for $r\ge 1$,
    \begin{align}
        r^{\frac{1}{2}+\frac{1}{\alpha}}\max_{i}(\sum_{j}a_{ij}^{2}p_{j})^{1/2}&\le (r^{1/\alpha}\Vert A\Vert_{\infty}^{1/2})\cdot(r^{1/2}\max_{i}(\sum_{j}\vert a_{ij}\vert p_{j})^{1/2})\nonumber\\
        &\le \frac{1}{2}r^{2/\alpha}\Vert A\Vert_{\infty}+\frac{1}{2}r\max_{i}\sum_{j}\vert a_{ij}\vert p_{j}\nonumber
    \end{align}
and 
\begin{align}
    \Vert (a_{ij}\sqrt{p_{i}p_{j}})\Vert_{l_{2}\to l_{2}}\le \max_{i}\sum_{j}\vert a_{ij}\vert \sqrt{p_{i}p_{j}}\le K\max_{i}\sum_{j}\vert a_{ij}\vert p_{j},\nonumber
\end{align}
where $K=K(p_{1}, \cdots,p_{n})=\max p_{i}/\min p_{j}$. Hence, Theorem \ref{Theo_main2} yields that 
\begin{align}
	\Big\Vert  \xi^\top A\xi-\mathbb{E}\xi^{\top}A\xi \Big\Vert_{L_{r}}\lesssim_{\alpha}&r^{2/\alpha}\Vert A\Vert_{\infty}+r^{1/2+1/\alpha}\max_{i} \Big(\sum_{j}a_{ij}^{2}p_{j} \Big)^{1/2}\nonumber\\
	&+r\Vert (a_{ij}\sqrt{p_{i}p_{j}})\Vert_{l_{2}\to l_{2}}+\sqrt{r}\Big(\sum_{k}a_{kk}^{2}p_{k}+\sum_{i\neq j}a_{ij}^{2}p_{i}p_{j} \Big)^{1/2}\nonumber\\
	\lesssim_{\alpha} &r^{2/\alpha}\Vert A\Vert_{\infty}+Kr\max_{i}\sum_{i}\vert a_{ij}\vert p_{j}+\sqrt{r}\Big(\sum_{k}a_{kk}^{2}p_{k}+\sum_{i\neq j}a_{ij}^{2}p_{i}p_{j} \Big)^{1/2}.\nonumber
\end{align}
Then, by Lemma \ref{Lem_Moments_to_tailbound} below, we deduce \eqref{Eq_Intro_He_Hw} from Theorem \ref{Theo_main2} but with an extra factor $K$.
	
\end{remark}

Note that, when $A$ is a diagonal-free matrix, the decay rate of the Hanson-Wright inequality \eqref{Eq_intro_logconvex} is optimal in the sense that, there exist independent centered $\alpha$-sub-exponential random variables $\eta_{1}, \cdots, \eta_{n}$ such that a matching lower bound holds.  See Proposition~\ref{Lem_appendix_1} for details and its proof is given in Appendix~\ref{appendix_A}.
\begin{prop}[Optimality of Theorem~\ref{Theo_main2} when $p_{1}=\cdots=p_{2}=1$]\label{Lem_appendix_1}
	Let $\eta_{1}, \cdots, \eta_{n}\stackrel{i.i.d.}{\sim}\mathcal{W}_{s}(\alpha)$ and $A$ be a symmetric diagonal-free matrix. Then for $0<\alpha\le 1$ and $t\ge 0$
	\begin{align}
		\mathbb{P}\{\vert \eta^\top A\eta-\mathbb{E}\eta^\top A\eta\vert\ge t    \}\ge C(\alpha)\exp(-c(\alpha)f_{2}(t)),\nonumber
	\end{align}
	where $\eta=(\eta_{1}, \cdots,\eta_{n})^{\top}$, $C(\alpha), c(\alpha)$ are  positive constants depending only on $\alpha$, and $f_{2}(t)$ is defined in \eqref{Eq_intro_f2}.
\end{prop}

\begin{remark}[Generalization to general $A$]
 Theorems \ref{Theo_main1} and \ref{Theo_main2} assume that the matrix $A$ is symmetric. This assumption was made primarily for the convenience of presentation. Indeed, the above results can be extended to general square matrices. The only modification required is that in many places, $A$ should be replaced by $\frac{1}{2}(A+A^{\top})$.
\end{remark}




\subsection{Disscussion}
In this section, we assume that $\{\delta_{i}, 1\le i\le n \}$ is a sequence of  independent random variables with $\delta_{i}\sim\textnormal{Bernoulli}(p_{i})$ and $\{\zeta_{i}, 1\le i\le n  \}$ is a sequence of independent centered $\alpha$-sub-exponential random variables. Assume $\delta_{i}$'s are independent of $\zeta_{i}$'s. 
Denote by $\xi=(\xi_{1}, \cdots, \xi_{n})$ the random vector  with $\xi_{i}=\delta_{i}\cdot\zeta_{i}$. Let $A=(a_{ij})_{n\times n}$ be a symmetric fixed matrix and $a=(a_{1}, \cdots, a_{n})^\top$ be a fixed vector. Set $L(\alpha)=\max_{i}\Vert \zeta_{i}\Vert_{\Psi_{\alpha}}$. We mention the following two open questions:

\begin{enumerate}
    \item For $1\le \alpha\le 2$ and $t>0$, it holds (see Theorem 6.1 in \cite{Adamczak_PTRF})
	\begin{align}
	\mathbb{P}\Big\{\big\vert \sum_{i=1}^{n}a_{i}\zeta_{i}\big\vert\ge t   \Big\}\le 2\exp\Big(-c(\alpha)\min\Big\{\frac{t^{2}}{L(\alpha)^{2}\Vert a\Vert_{2}^{2}}, \Big(\frac{t}{L(\alpha)\Vert a\Vert_{\alpha^{*}}}\Big)^{\alpha}  \Big\}   \Big)\nonumber
\end{align}
and 
\begin{align}
	\mathbb{P}\{\vert \zeta^\top A\zeta-\mathbb{E}\zeta^\top A\zeta\vert\ge L(\alpha)^{2}t    \}\le 2\exp(-c(\alpha)f_{2}(t)),\nonumber
\end{align}
where  $\alpha^{*}=\alpha/(\alpha-1)$ and $f_{1}(t)$ is defined in \eqref{Eq_intro_f1}.
How to obtain upper bounds for
\begin{align}
	\mathbb{P}\Big\{\big\vert \sum_{i=1}^{n}a_{i}\xi_{i}\big\vert\ge t   \Big\},\quad \mathbb{P}\{\vert \xi^\top A\xi-\mathbb{E}\xi^\top A\xi\vert\ge t    \}\nonumber
\end{align}
which recover the above results when $\delta_{1}=\cdots\delta_{n}=1$ is an interesting question.
\item 
In this paper, we give upper bounds for the linear and quadratic forms of sparse $\alpha$-sub-exponential random vectors:
\begin{align}
	\big\Vert \sum_{i=1}^{n} a_{i}\xi_{i}\big\Vert_{L_{r}}, \quad \Vert \xi^{\top}A\xi-\mathbb{E}\xi^{\top}A\xi\Vert_{L_{r}}.\nonumber
\end{align}
Investigating their lower bounds is an interesting question, which may help us establish the optimal bounds for the above quantities.
\end{enumerate}


\section{Applications}
In this section, we present several applications in covariance estimation, randomized low-rank approximation, and concentration of random vectors.
Sparsity can naturally come from the fact that  a high-dimensional random vector  is sparse, for example, when the elements of the vector are missing at random, or when we intentionally sparsify the vector to speed up computation.

\subsection{Covariance matrix estimation under missing observations}

Concentration properties of the sample covariance matrix play a central role in numerous high-dimensional inference tasks, including sparse PCA, graphical model selection, discriminant analysis, and high-dimensional regression \cite{CK22}. 

We consider the problem of estimating the covariance matrix of a random vector from i.i.d.\ samples, where each vector is observed with entries missing independently.  In the sub-Gaussian setting, this problem was first analyzed in \cite{lounici2014high}, motivated by applications in climate science, gene expression, and cosmology, where an unbiased estimator derived from the sample covariance matrix was proposed. The corresponding error bounds were later sharpened in \cite{klochkov2020uniform}. More recently, under the $L_4$--$L_2$ moment equivalence assumption, Abdalla \cite{abdalla2024covariance} introduced an estimator that achieves both optimal dependence on the missingness parameter~$p$ and the minimax optimal convergence rate under the spectral norm. However, this estimator is computationally intractable in practice.  

Beyond these contributions, a number of related results on covariance estimation with missing data have appeared in the literature, primarily for sub-Gaussian random vectors; see, for example, \cite{cai2016minimax,pavez2020covariance,park2021estimating,kolar2012estimating,Zhou_Bernoulli,Zhou_Bernoulli_2024}. In this work, we go beyond the sub-Gaussian framework and develop  concentration results for covariance estimation under heavy-tailed distributions with missing observations.


We state our assumptions as follows:
\begin{assumption}[Data assumption]
Assume that the entries of $\xi\in \R^m$ are independent, centered, $\alpha$-sub-exponential random variables with $\alpha \in (0,2]$ with variance $1$. Assume \[Y=B\xi\] where \(B\) is a fixed \(d\times m\) matrix and $BB^\top =\Sigma$.
\end{assumption}
This means the data vector $Y$ satisfies a multivariate model \cite{fan2008high,cai2016minimax} as a linear transformation of $\xi^{(s)}$, which is a common assumption in high-dimensional statistics.
We assume the data vector is observed partially with missing entries, defined as follows:
\begin{assumption}[Missing at random]
      We observe \(n\) independent samples \(X^{(1)},\ldots,X^{(n)}\in \R^d\), where each  
\[
X^{(s)}= \delta^{(s) } \circ   Y^{(s)},
\]  
where 
\begin{itemize}
    \item \(\delta^{(s)}\in\{0,1\}^{d}\) are i.i.d. Bernoulli random vectors with independent entries \(\delta^{(s)}_{i}\sim\mathrm{Bernoulli}(p_i)\), \(i=1,\dots,d\);
    \item \(Y^{(s)}\in\mathbb{R}^{m}\) are i.i.d. copies of a random vector $Y$.
\end{itemize}
\end{assumption}
 


The goal is to estimate  $BB^\top=\Sigma$. A natural and unbiased estimator of $\Sigma$ is defined by

\begin{align*}
    \widehat{\Sigma}_{jk} =\begin{cases}
        \frac{1}{np_j p_k}\sum_{s=1}^n
     X_j^{(s)}X_k^{(s)}, & \quad j\not=k,\\
     \frac{1}{np_j} \sum_{s=1}^n {(X_j^{(s)})}^2 & \quad j=k.
    \end{cases} 
\end{align*}
This estimator is called the \textit{inverse probability weighting} (IPW) estimator \cite{kolar2012estimating,friedman2008sparse,cai2016minimax,park2021estimating}. When $Y$ is a sub-Gaussian random vector, the element-wise deviation of the IPW estimator has been considered in \cite{kolar2012estimating,lounici2014high,park2021estimating}.



As soon as the dimension $d$ exceeds the sample size $n$, however, the sample covariance matrix ceases to be consistent in the operator norm.    A standard remedy is to restrict attention to the “significant” coordinates, i.e. to work with a sparse sub-vector.  
Motivated by this practice we analyse the maximum $k$-sparse sub-matrix operator norm of the covariance matrix \cite{CK22}. In particular, we would like to bound the following quantity:
\[\mathrm{RIP}_n(k):= \sup_{\substack{\theta\in\mathbb{R}^d, \|\theta\|_2\le 1\\ \|\theta\|_0\le k}} \left| \theta^\T(\widehat{\Sigma} - \Sigma)\theta\right|.\]

This norm is also of importance in high-dimensional
linear regression due to its connections to the restricted isometry property (RIP) \cite{candes2007dantzig} and the restricted eigenvalue (RE) condition \cite{bickel2009simultaneous}.
When $k=1$, $\mathrm{RIP}_n(1)$ is $\|\hat{\Sigma}-\Sigma\|_{\infty}$ and when $k=d$, $\mathrm{RIP}_n(1)$ is the operator norm  $\|\hat{\Sigma}-\Sigma\|_{l_2\to l_2}$.

Before giving the main result of this subsection, we introduce some notations.  For the vectors $p=(p_1,\cdots, p_d)^\T$ and $\theta=(\theta_1,\cdots, \theta_d)^\T$, the vector $\theta\circ\frac{1}{p}$ is defined as follows:
\begin{align}
    \theta\circ\frac{1}{p}=(\theta_1/p_1, \cdots, \theta_d/p_d)^\T.
\end{align}
We define the matrix  $A_{\theta, p}\in\mathbb{R}^{d\times d}$ as follows:
    \[
(A_{\theta, p})_{jk}=
\begin{cases}
\theta_j^2/p_j, & j=k,\\[4pt]
\dfrac{\theta_j\theta_k}{p_j p_k}, & j\neq k.
\end{cases}
\]

\begin{theorem}\label{thm:RIP}
    Assume that $\max_{i}\Vert \xi_{i}\Vert_{\Psi_{\alpha}}\le 1$. For any $t\ge 0$, with probability at least $1-2e^{-t}$,
    \begin{align}      \mathrm{RIP}_n(k)&\lesssim_{\alpha}\frac{(t+k\log(48ed/k))^{1/2}}{\sqrt{n}}\sup_{\theta\in\Omega}\Big(\mathbb{E}\Vert B^\T\mathrm{Diag}(\delta^{(1)})A_{\theta, p}\mathrm{Diag}(\delta^{(1)})B\Vert_{F}^2\Big)^{1/2}\nonumber\\&+\Big(\frac{(t+k\log(48ed/k))^{3/4}}{n^{3/4}}+\frac{(t+k\log(48ed/k))^{2/\alpha}}{n}  \Big)\times\nonumber\\
    &\quad\Vert B\Vert_{l_2\to l_2}\cdot(\mathrm{Diag}(\sqrt{p_1}, \cdots, \sqrt{p_d} )\cdot B\Vert_{F}+\Vert B\Vert_{l_2\to l_2})\cdot\sup_{\theta\in\Omega}\Vert\theta\circ\frac{1}{p}\Vert^2_{2},
    \end{align}
   where $\Omega=\{\theta\in\mathbb{R}^{d}: \Vert \theta\Vert_{2}\le 1, \Vert \theta\Vert_{0}\le k\}$.
\end{theorem}

\begin{remark}[Comparison with \cite{KC22}]
 When all \( p_i=1 \), \cite{KC22} studied covariance estimation for marginally sub-Weibull random vectors. Compared with their result, Theorem \ref{thm:RIP} removes the extra \( \log n \) factor of the $t^{2/\alpha}$ term, and extends the analysis to the missing-observation setting; however, our assumption on the data \( Y \) is stronger, whereas \cite{KC22} does not require \( X \) to follow a multivariate model.
\end{remark}
\begin{remark}
    (i). By Propositions \ref{Proposition_application_covariace_1} and \ref{Proposition_application_covariace_2} below, we have
    \begin{align}
        &\Big(\mathbb{E}\Vert B^\T\mathrm{Diag}(\delta^{(1)})A_{\theta, p}\mathrm{Diag}(\delta^{(1)})B\Vert_{F}^2\Big)^{1/2}\nonumber\\
        \lesssim& \Vert B\Vert_{l_2\to l_2}\cdot( \|\mathrm{Diag}(\sqrt{p_1}, \cdots, \sqrt{p_d} )\cdot B\Vert_{F}+\Vert B\Vert_{l_2\to l_2})\cdot\Vert\theta\circ\frac{1}{p}\Vert^2_{2}.
    \end{align}

    (ii). An exact calculation of $\mathbb{E}\Vert B^\T\mathrm{Diag}(\delta^{(1)})A_{\theta, p}\mathrm{Diag}(\delta^{(1)})B\Vert_{F}^2$ is given in the Appendix \ref{appendix_e}.
\end{remark}


\subsection{Low-rank approximation with sparsified sketches}
Low-rank approximation is fundamental in data analysis, signal processing, and machine learning. Many large data matrices, such as user–item ratings, gene expression profiles, or kernel matrices, are approximately low-rank due to latent factors or redundancy~\cite{udell2019big}. Exploiting this structure enables efficient algorithms and reduces storage and computation in downstream tasks.

Recent work~\cite{udell2019big,Stanislav_LAA} shows that large matrices with small spectral norm admit nearby low-rank approximations in the \emph{maximum norm}, providing uniform entrywise guarantees. Building on this perspective, we propose a randomized construction using \emph{sparsified random vectors}. Instead of dense Gaussian sketches, we use probes of the form $\delta \circ \xi$, where $\delta$ is a Bernoulli mask and $\xi$ is sub-Gaussian. This reduces storage and matrix–vector costs while introducing dependence on the missingness probability $p$. Our sparse Hanson–Wright inequalities yield sharp analysis, showing that sparsification preserves robust max-norm guarantees with explicit dependence on matrix coherence.

Sparsified sketches have also appeared in stochastic trace estimation~\cite{kalantzis2024asynchronous}, analyzed via sparse Hanson–Wright inequalities~\cite{Zhou_Bernoulli}. Our results extend this principle to low-rank approximation, demonstrating both computational savings and theoretical control.


Before stating the main result in this subsection we fix notation.  
For a rank-$k$ matrix $X\in\mathbb{R}^{m\times n}$ with thin SVD $X=U\Sigma V^{\top}$, the column and row-space coherences are

\[
\mu_{\mathrm{col}}(X)=\frac{m}{k}\max_{i\in[m]}\|U^{\top}e_{i}\|_{2}^{2}, \qquad
\mu_{\mathrm{row}}(X)=\frac{n}{k}\max_{j\in[n]}\|V^{\top}\tilde e_{j}\|_{2}^{2},
\]
where $e_{i}\!\in\!\mathbb{R}^{m}$ and $\tilde e_{j}\!\in\!\mathbb{R}^{n}$ are the standard basis vectors.

We propose a randomized algorithm to construct a low-rank approximation of a rank-$k$ matrix $X$ given in Algorithm~\ref{alg:low_rank}. The analysis of Algorithm~\ref{alg:low_rank} uses the sparse Hanson-Wright inequality for a sparse sub-Gaussian random vector.

\begin{algorithm}[h!]
\caption{Low rank approximation with sparsified sketches}
\begin{algorithmic}[1]
\INPUT A rank-$k$ matrix $X\in \R^{m\times n}$, target rank $r$, sparsity parameter $p$. 
\STATE Compute singular value decomposition $X=U\Sigma V^\top$ where $U\in\mathbb{R}^{m\times k}, \Sigma\in \mathbb{R}^{k\times k}$ and $V\in \mathbb{R}^{n\times k}$.  Define \[\widetilde{U}=U\Sigma^{1/2}=[\tilde{u}^\T_1, \cdots, \tilde{u}^\T_m]^\T,\quad \widetilde{V}=V\Sigma^{1/2}=[\tilde{v}^\top_1, \cdots, \tilde{v}^\top_n]^\top.\]
\STATE  Let $Q=[Q_1, \cdots, Q_r]\in\mathbb{R}^{k\times r}$ be a random matrix with independent copies of $r^{-\frac{1}{2}}\delta\circ\xi$, where $\delta\sim \mathrm{Bernoulli}(p)$ and $\xi$ is a centered sub-gaussian random variable satisfying $\Vert \xi\Vert_{\Psi_2}\le 1$ and $\mathbb{E}\xi^2=1$. The random variable $\delta$ is independent of $\xi$.
\OUTPUT Matrix $Y\in \R^{m\times n}$ such that $Y_{ij}=\frac{1}{p} \tilde u_i QQ^\top \tilde v_j^\top$.
\end{algorithmic}\label{alg:low_rank}
\end{algorithm}

\begin{theorem}\label{Theo_application_low}
    Assume that $m, n, k\in\mathbb{N}$ satisfies $k\le \min\{m, n\}$. Let $0<p\le 1$ and $\varepsilon>0$. Then for every matrix $X\in\mathbb{R}^{m\times n}$ of rank $k$ and any constant $r$ such that $$k\ge r\ge C_1\log \frac{mn}{\eta}\cdot\max\{ \frac{ p}{\varepsilon^2}, \frac{1 }{\varepsilon} \},$$ then Algorithm~\ref{alg:low_rank} outputs a matrix $Y\in\mathbb{R}^{m\times n}$ of rank no more than $r$ such that with probability at least $1-\eta$,
    \begin{align}
        \Vert X-Y\Vert_{\infty}\le \frac{k\varepsilon}{p\sqrt{mn}}\sqrt{\mu_\mathrm{col}\mu_\mathrm{row}}\cdot\Vert X\Vert_{l_2\to l_2},
    \end{align}
    where $\mu_\mathrm{col}$ and $\mu_\mathrm{row}$ are the coherence of the column and row spaces of $X$.
\end{theorem}
\begin{remark}
 Recently, Budzinskiy \cite[Theorem 3.4]{Stanislav_LAA} proved that under the condition $k\ge r\ge \frac{C_1 \log mn}{\varepsilon^2}$,
     there exists a matrix $Y\in\mathbb{R}^{m\times n}$ of rank no more than $r$ such that
    \begin{align}
        \Vert X-Y\Vert_{\infty}\le \frac{k\varepsilon}{\sqrt{mn}}\sqrt{\mu_\mathrm{col}\mu_\mathrm{row}}\cdot\Vert X\Vert_{l_2\to l_2}.
    \end{align}
    Relative to Budzinskiy’s theorem, Theorem \ref{Theo_application_low} widens the admissible range of $r$ at the cost of a looser error bound. In practice, we can trade off rank versus accuracy according to the task at hand.  Our proof shows that $Y$ can be constructed based on a significantly sparser matrix $Q$, yielding substantial savings in storage and matrix–vector products compared with the dense approximation proposed in \cite{Stanislav_LAA}.
\end{remark}

\subsection{Sparse \texorpdfstring{$\alpha$}{alpha}-sub-exponential concentration}

In this subsection we establish a concentration inequality for random vectors with independent $\alpha$-sub-exponential coordinates; such bounds are widely used \cite{Rudelson_ECP}.  The derivation is standard, so the proof is relegated to the appendix. 

\begin{theorem}\label{Theorem_application}
    Assume that $\{\delta_{i}, 1\le i\le n \}$ and $\{\zeta_{i}, 1\le i\le n \}$ are two independent sequences of random variables, where $\delta_{i}\sim\textnormal{Bernoulli}(p)$ are independent and $\zeta_{i}$ are independent $\alpha$-sub-exponential random variables with mean $0$ and variance $1$.
    Consider a random vector $\xi=(\xi_{1}, \cdots, \xi_{n})$  with $\xi_{i}=\delta_{i}\cdot\zeta_{i}$. Let $A$ be an $m\times n$ fixed matrix. Then for $0<\alpha\le 2$ and $t\ge 0$,
    \begin{align}
        \mathbb{P}\Big\{ \big\vert \Vert A\xi\Vert_{2}-\sqrt{p}\Vert A\Vert_{F} \big\vert>L^{2}(\alpha)t \Big\}\le 2\exp\Big\{-c(\alpha)\min \big\{\big(\frac{t}{\Vert A\Vert_{l_2\to l_2}}\big)^{2}, \big(\frac{t}{\Vert A\Vert_{l_2\to l_2}} \big)^{\alpha} \big\}    \Big\},
    \end{align}
    where $L(\alpha)=\max_{i}\Vert \zeta_i\Vert_{\Psi_{\alpha}}$ and $c(\alpha)>0$ is a constant depending only on $\alpha$.
\end{theorem}
\begin{remark}
    (i). In the case $0<\alpha\le 1$ and $p=1$, G\"{o}tze et al. \cite[Proposition 2.1]{Gotze_EJP} proved that
    \begin{align}
        &\mathbb{P}\Big\{ \big\vert \Vert A\xi\Vert_{2}-\Vert A\Vert_{F} \big\vert>L^{2}(\alpha)t \Big\}\\
        \le& 2\exp\Big\{-c(\alpha)\min \big\{\frac{t^2}{\Vert A\Vert^{2-\alpha}_F\Vert A\Vert^{\alpha}_{l_2\to l_2}}, \big(\frac{t}{\Vert A\Vert_{l_2\to l_2}} \big)^{\alpha} \big\}    \Big\}.
    \end{align}
    Compared with their result, Theorem \ref{Theorem_application} gives an improved tail probability due to that
    \begin{align}
        \Vert A\Vert_{l_2\to l_2}^2\le \Vert A\Vert^{2-\alpha}_F\Vert A\Vert^{\alpha}_{l_2\to l_2}.
    \end{align}

    (ii). In the case $\alpha=2$ and $p=1$, Rudelson and Vershynin \cite[Theorem 2.1]{Rudelson_ECP} proved that
    \begin{align}\label{Eq_application_remark_1}
        \mathbb{P}\Big\{ \big\vert \Vert A\xi\Vert_{2}-\Vert A\Vert_{F} \big\vert>L^{2}(\alpha)t \Big\}
        \le 2\exp\Big\{-\frac{ct^2}{\Vert A\Vert^{2}_{l_2\to l_2}}    \Big\}.
    \end{align}
    Theorem \ref{Theorem_application} recovers this result in the case $\alpha=2$ and $p=1$.
\end{remark}

\begin{remark}
When \(\alpha = 2\), we have the following inequality:
\[
\|\xi_i\|_{\Psi_2} \leq \|\delta_i\|_{\Psi_2} \|\zeta_i\|_{\Psi_2} \leq L \log^{-1/2}(1 + p^{-1}),
\]
where \(L = \max_i \|\zeta_i\|_{\Psi_2}\). This implies that \(\xi_i\) is a sub-Gaussian random variable. Consequently, we can derive a result of the same type as in Theorem \ref{Theorem_application} from \eqref{Eq_application_remark_1}.

In particular, Rudelson and Vershynin's result gives us:
\[
\mathbb{P}\left\{\left|\|A\xi\|_2^2 - p \|A\|_F^2\right| > L^2 t\right\} \leq 2 \exp\left\{-\frac{c \log^2(1 + p^{-1}) p t^2}{\|A\|_{l_2 \to l_2}^2}\right\}.
\]

Compared to this result, Theorem \ref{Theorem_application} provides a better tail probability when \(p\) is close to 0. This serves as a good example to illustrate why we study sparse Hanson-Wright type inequalities.
\end{remark}

\section{Preliminaries}

\subsection{Tails and Moments}
In this subsection, we first introduce
the following lemma, which provides the link between $L_{r}$ estimates and tail probability inequalities. Although such type results are by now common, we give a brief proof for the sake of reading.
\begin{lemma}\label{Lem_Moments_to_tailbound}
	Let $\xi$ be a random variable such that for $r\ge r_{0}$
	\begin{align}
		\Vert \xi\Vert_{L_r}\le \sum_{k=1}^{m}C_{k}r^{\beta_{k}}+C_{m+1},\nonumber
	\end{align}
	where $C_{1},\cdots, C_{m+1}> 0$ and $\beta_{1},\cdots, \beta_{m}>0$. Then for $t>0$,
	\begin{align}
		\mathbb{P}\big\{ \vert \xi\vert>e(mt+C_{m+1}) \big\}\le e^{r_{0}}\exp\Big(-\min\big\{\big(\frac{t}{C_{1}}\big)^{1/\beta_{1}},\cdots, \big(\frac{t}{C_{m}}\big)^{1/\beta_{m}}\big\}\Big).\nonumber
	\end{align}
\end{lemma}

\begin{proof}
	Define the following function:
	\begin{align}
		f(t):=\min\big\{\big(\frac{t}{C_{1}}\big)^{1/\beta_{1}},\cdots, \big(\frac{t}{C_{m}}\big)^{1/\beta_{m}}\big\}.\nonumber
	\end{align}
	If $f(t)\ge r_{0}$, then
	\begin{align}
		\Vert \xi\Vert_{L_{f(t)}}\le \sum_{k=1}^{m}C_{k}f(t)^{\beta_{k}}+C_{m+1}\le mt+C_{m+1}.\nonumber
	\end{align}
	Hence,  Markov's inequality yields
	\begin{align}
		\mathbb{P}\{\vert \xi\vert>e(mt+C_{m+1})  \}\le \mathbb{P}\{\vert \xi\vert >e \Vert \xi\Vert_{L_{f(t)}}  \}\le e^{-f(t)}.\nonumber
	\end{align}
	As for $f(t)<r_{0}$, we have the following trivial bound
	\begin{align}
		\mathbb{P}\{\vert \xi\vert>e(mt+C_{m+1})  \}\le 1\le e^{r_{0}}e^{-f(t)}.\nonumber
	\end{align}
	Hence, we have for $t\ge 0$
	\begin{align}
		\mathbb{P}\{\vert \xi\vert>e(mt+C_{m+1})  \}\le  e^{r_{0}}e^{-f(t)}.\nonumber
	\end{align}
\end{proof}
The next lemma estimates the $r$-th moments of random linear and bilinear sums.
\begin{lemma}[Theorem 1.1 in \cite{Hitczenko_studiamath}]\label{Lem_Lp_linearsum}
	Let $\xi_{1}, \cdots, \xi_{n}$ be independent symmetric random variables with log-convex tails, i.e., $\log \mathbb{P}\{\vert\xi_{i}\vert\ge t \}$ is a convex function for $t\ge 0$. Then for $r\ge 2$
	\begin{align}
		\Big\Vert \sum_{i}\xi_{i}\Big\Vert_{L_{r}}\asymp  \Big(\sum_{i}\mathbb{E}\vert \xi_{i}\vert^{r}   \Big)^{1/r}+\sqrt{r}\Big( \sum_{i}\mathbb{E}\xi_{i}^{2}  \Big)^{1/2}.\nonumber
	\end{align}
\end{lemma}
In the following lemmas of this subsection, we assume that $\xi_{1}, \cdots, \xi_{n}\stackrel{i.i.d.}{\sim}\mathcal{W}_{s}(\alpha)$ and $A=(a_{ij})_{n\times n}$ is a fixed symmetric diagonal-free matrix. Assume further that $\tilde{\xi}_{i}$ is an independent copy of $\xi_{i}$, $1\le i\le n$. For any $1\le \alpha\le 2$, let $\alpha^{*}$ be its conjugate exponent, i.e., $1/\alpha+1/\alpha^{*}=1$.
\begin{lemma}[Theorem 6.1 in \cite{Adamczak_PTRF}]\label{Lem_Lp_quadratic_1_2}
	 In the case $1\le \alpha\le 2$, we have  for $r\ge 2$
\begin{align}
		\Big\Vert \sum_{i,j}a_{ij}\xi_{i}\tilde{\xi}_{j}\Big\Vert_{L_{r}}\asymp_{\alpha}&r^{1/2}\Vert A\Vert_{F}+r\Vert A\Vert_{l_{2}\to l_{2}}+r^{1/\alpha}\Vert A\Vert_{l_{\alpha^{*}}(l_{2})}\nonumber\\
		&+r^{(\alpha+2)/2\alpha}\Vert A\Vert_{l_{2}\to l_{\alpha^{*}}}+r^{2/\alpha}\Vert A\Vert_{l_{\alpha}\to l_{\alpha^{*}}},\nonumber
	\end{align}
    where $\Vert A\Vert_{l_{\alpha*}(l_2)}$ is defined in \eqref{matrix_norm}.
\end{lemma}
\begin{lemma}[Example 3 in \cite{Kolesko_SPL}]\label{Lem_Lp_quadratic_0_1}
	In the case $0< \alpha\le 1$, we have  for $r\ge 2$
	\begin{align}
		\Big\Vert \sum_{i,j}a_{ij}\xi_{i}\tilde{\xi}_{j}\Big\Vert_{L_{r}}\asymp_{\alpha}&r^{1/2}\Vert A\Vert_{F}+r\Vert A\Vert_{l_{2}\to l_{2}}+r^{(\alpha+2)/2\alpha}\Vert A\Vert_{l_{2}\to l_{\infty}}+r^{2/\alpha}\Vert A\Vert_{\infty}.\nonumber
	\end{align}
\end{lemma}
By employing a similar argument as in \cite{Sambale_Progress_probability}, we can derive the following simplified result from Lemmas \ref{Lem_Lp_quadratic_1_2} and \ref{Lem_Lp_quadratic_0_1}. For ease of reference, the proof is provided in Appendix \ref{appendix_B}.
\begin{coro}\label{Cor_Section2_quadratic_forms}
	In the case $0<\alpha\le 2$, we have for $r\ge 2$,
	\begin{align}
		\Big\Vert \sum_{i,j}a_{ij}\xi_{i}\tilde{\xi}_{j}\Big\Vert_{L_{r}}\lesssim_{\alpha}r^{1/2}\Vert A\Vert_{F}+r^{2/\alpha}\Vert A\Vert_{l_{2}\to l_{2}}.\nonumber
	\end{align}
\end{coro}

\subsection{Decoupling inequality}

Decoupling is a technique of replacing quadratic forms of random variables by bilinear forms. The monograph \cite{de_la_Pena_Book} systematically studies decoupling and its applications. In this subsection, we introduce a classic decoupling inequality as follows:

\begin{lemma}[Theorem 3.1.1 in \cite{de_la_Pena_Book}]\label{Lem_decoupling_inequality}
	Let $F: \mathbb{R}^{+}\to \mathbb{R}^{+}$ be a convex function and $A=(a_{ij})_{n\times n}$ be a diagonal-free matrix. If $\{\xi_{i}, i\le n\}$ is a sequence of independent random variables and $\tilde{\xi}_{i}$ is an independent copy of $\xi_{i}$, then there exists a universal constant $C$ such that
	\begin{align}\label{Eq_Section2_lem_decouple}
		\mathbb{E}F\Big(\big\vert\sum_{i,j}a_{ij}\xi_{i}\xi_{j}\big\vert   \Big)\le \mathbb{E}F\Big(C\big\vert\sum_{i,j}a_{ij}\xi_{i}\tilde{\xi}_{j}   \big\vert\Big).
	\end{align} 
\end{lemma}
\begin{remark}
	If, moreover, $A=(a_{ij})$ is a symmetric diagnal-free matrix, then \eqref{Eq_Section2_lem_decouple} can be reversed, that is,
\begin{align}
	 \mathbb{E}F\Big(\frac{1}{C}\big\vert\sum_{i,j}a_{ij}\xi_{i}\tilde{\xi}_{j}   \big\vert\Big)\le \mathbb{E}F\Big(\big\vert\sum_{i,j}a_{ij}\xi_{i}\xi_{j}\big\vert   \Big).\nonumber
\end{align} 

\end{remark}

\subsection{Contraction principle}

In this subsection, we present a well-known contraction principle in Banach space. This principle enables the extension of results from Weibull random variables to $\alpha$-sub-exponential random variables.

\begin{lemma}[Lemma 4.6 in \cite{Ledoux_Book}]\label{Lem_contraction_principle}
	Let $F: \mathbb{R}^{+}\to \mathbb{R}^{+}$ be a convex non-decreasing function. Assume that $\{\eta_{i}, i\le n\}$ and $\{\xi_{i}, i\le n\}$ are two  sequences of independent symmetric random variables such that for some constant $K\ge 1$
	\begin{align}\label{Eq_Section2_contraction}
		\mathbb{P}\{\vert \eta_{i}\vert>t  \}\le K\mathbb{P}\{\vert\xi_{i}\vert >t \},\quad i\le n,\quad t>0.
	\end{align}
	Then, for any sequence $\{a_{i}, i\le n\}$ in a Banach space $(\mathcal{A}, \Vert\cdot\Vert)$,
	\begin{align}
		\mathbb{E}F\Big( \big\Vert \sum_{i=1}^{n}a_{i}\eta_{i}\big\Vert    \Big)\le \mathbb{E}F\Big(K \big\Vert \sum_{i=1}^{n}a_{i}\xi_{i}\big\Vert   \Big).\nonumber
	\end{align}
	
\end{lemma}
\begin{remark}\label{Rem_contraction_principle}
	If $\{\eta_{i}, i\le n\}$ and $\{\xi_{i}, i\le n\}$ are two  sequences of independent centered random variables satisfying \eqref{Eq_Section2_contraction}, the result in Lemma \ref{Lem_contraction_principle} is still valid by a symmetrization argument. Indeed, let $\tilde{\eta}_{i}$ be  independent copy of $\eta_{i}$, $1\le i\le n$. Then, Jensen inequality yields
	\begin{align}
		\mathbb{E}F\Big( \big\Vert \sum_{i=1}^{n}a_{i}\eta_{i}\big\Vert    \Big)&=\mathbb{E}F\Big( \big\Vert \sum_{i=1}^{n}a_{i}(\eta_{i}-\mathbb{E}_{\tilde{\eta}}\tilde{\eta}_{i})\big\Vert    \Big)\nonumber\\
		&\le \mathbb{E}F\Big( \big\Vert \sum_{i=1}^{n}a_{i}(\eta_{i}-\tilde{\eta}_{i})\big\Vert    \Big)\le \mathbb{E}F\Big( 2\big\Vert \sum_{i=1}^{n}a_{i}\varepsilon_{i}\eta_{i}\big\Vert    \Big),\nonumber
	\end{align}
where $\{\varepsilon_{i}, i\le n\}$ are independent Rademacher variables. On the other hand,  by the convexity of $F$
\begin{align}
	\mathbb{E}F\Big( 2\big\Vert \sum_{i=1}^{n}a_{i}\varepsilon_{i}\eta_{i}\big\Vert    \Big)\le \mathbb{E}F\Big( 2\big\Vert \sum_{i=1}^{n}a_{i}\varepsilon_{i}(\eta_{i}-\tilde{\eta}_{i})\big\Vert    \Big)\le \mathbb{E}F\Big( 4\big\Vert \sum_{i=1}^{n}a_{i}\eta_{i}\big\Vert    \Big).\nonumber
\end{align}
Hence, we have 
\begin{align}
	\mathbb{E}F\Big( \big\Vert \sum_{i=1}^{n}a_{i}\eta_{i}\big\Vert    \Big)\le \mathbb{E}F\Big( 2\big\Vert \sum_{i=1}^{n}a_{i}\varepsilon_{i}\eta_{i}\big\Vert    \Big)\le \mathbb{E}F\Big( 4\big\Vert \sum_{i=1}^{n}a_{i}\eta_{i}\big\Vert    \Big).\nonumber 
\end{align}
\end{remark}

\subsection{Concentration inequality}
In this subsection, we shall introduce Talagrand's concentration inequality for convex Lipschitz functions.
\begin{lemma}[Corollary 4.10 in \cite{Ledoux_Book_concentration_measure}]\label{Lem_concentration_inequality}
	Let $\xi_{1}, \cdots, \xi_{n}$ be independent random variables such that $\vert \xi_{i}\vert\le K$ for all $1\le i\le n$. Assume that $f: \mathbb{R}^{n}\to \mathbb{R}$ is a convex and $1$-Lipschitz function. Then for $t>0$
	\begin{align}
		\mathbb{P}\big\{\big\vert f(\xi_{1},\cdots, \xi_{n})-\mathbb{E}f(\xi_{1}, \cdots, \xi_{n}) \big\vert>Kt  \big\}\le 4e^{-t^{2}/4}.\nonumber
	\end{align}

\end{lemma}
\subsection{The sparse Bernstein inequality}
In this subsection, we introduce the following sparse Bernstein inequality
in $\alpha$-sub-exponential random variables, $0<\alpha\le 1$. With this property, one can estimate the diagonal sum of the quadratic forms. 
\begin{lemma}\label{Prop_Section3_Bernstein_sparse}
Assume that $\{\delta_{i}, 1\le i\le n \}$ and $\{\zeta_{i}, 1\le i\le n \}$ are two independent sequences of random variables, where $\delta_{i}\sim\textnormal{Bernoulli}(p_{i})$ are independent and $\zeta_{i}$ are independent centered $\alpha$-sub-exponential random variables.
       Let $a=(a_{1},\cdots, a_{n})$ be a  fixed vector. Then for $0<\alpha\le 1$ and $t\ge 0$
       \begin{align}
       	\mathbb{P}\Big\{\big\vert \sum_{i=1}^{n}a_{i}\delta_{i}\zeta_{i}\big\vert\ge t   \Big\}\le 2\exp\Big(-c(\alpha)\min\Big\{\frac{t^{2}}{L(\alpha)^{2}(\sum_{i}a_{i}^{2}p_{i})}, \Big(\frac{t}{L(\alpha)\Vert a\Vert_{\infty}}\Big)^{\alpha}  \Big\}   \Big),\nonumber
       \end{align}
   where $L(\alpha)=\max_{i}\Vert \zeta_{i}\Vert_{\Psi_{\alpha}}$ and $c(\alpha)$ is a positive constant depending only on $\alpha$.
\end{lemma}
\begin{remark}
	(i). This property, though likely known in the literature, was explicitly stated and proved using a combinatorial approach in He et al. (see Theorem 3 in \cite{He_arXiv_HW}). In Appendix \ref{appendix_C} we present an alternative, more concise proof. 
	
	(ii). For $0<\alpha\le 1$ and $t\ge 0$, it holds (see Corollary 1.4 in \cite{Gotze_EJP})
	\begin{align}
		\mathbb{P}\Big\{\big\vert \sum_{i=1}^{n}a_{i}\zeta_{i}\big\vert\ge t   \Big\}\le 2\exp\Big(-c(\alpha)\min\Big\{\frac{t^{2}}{L(\alpha)^{2}\Vert a\Vert_{2}^{2}}, \Big(\frac{t}{L(\alpha)\Vert a\Vert_{\infty}}\Big)^{\alpha}  \Big\}   \Big).\nonumber
	\end{align}
Hence, Lemma \ref{Prop_Section3_Bernstein_sparse} recovers this classical result when $p_{1}=\cdots=p_{n}=1$.
\end{remark}

\section{Proofs of main results}
\subsection{The sparse Hanson-Wright inequality for $0<\alpha\le 2$}

We begin by proving Theorem \ref{Theo_main1} for the special case where matrix $A$ has zero diagonal elements. 
\begin{prop}\label{Prop_Section3_diagonal_free}
	Assume that $\{\delta_{i}, 1\le i\le n \}$ and $\{\zeta_{i}, 1\le i\le n \}$ are two independent sequences of random variables, where $\delta_{i}\sim\textnormal{Bernoulli}(p_{i})$ are independent and $\zeta_{i}$ are independent centered $\alpha$-sub-exponential random variables.
	Consider a random vector $\xi=(\xi_{1}, \cdots, \xi_{n})$  with $\xi_{i}=\delta_{i}\cdot\zeta_{i}$. Let $A=(a_{ij})_{n\times n}$ be a symmetric diagonal-free matrix. Then for $0<\alpha\le 2$ and $t\ge 0$, 
	\begin{align}
	\mathbb{P}\{\vert \xi^\top A\xi\vert\ge t  \}
		\le 2\exp\Big(-c(\alpha)\min\Big\{ \frac{t^{2}}{L(\alpha)^{4}\sum_{i, j}a_{ij}^{2}p_{i}p_{j}}, \big(\frac{t}{L(\alpha)^{2}\Vert A\Vert_{l_{2}\to l_{2}}}\big)^{\alpha/2}  \Big\}      \Big),\nonumber
	\end{align}
	where $L(\alpha)=\max_{i}\Vert \zeta_{i}\Vert_{\Psi_{\alpha}}$ and $c(\alpha)$ is a positive constant depending only on $\alpha$.
\end{prop}
\begin{proof}
	Let $\tilde{\xi}_{i}=\tilde{\delta}_{i}\cdot\tilde{\zeta}_{i}$ for $1\le i\le n$, where $\{\tilde{\delta}_{i}, 1\le i\le n\}$ and $\{\tilde{\zeta}_{i}, 1\le i\le n  \}$ are independent copies of $\{\delta_{i}, 1\le i\le n\}$ and $\{\zeta_{i}, 1\le i\le n  \}$, respectively. Lemma \ref{Lem_decoupling_inequality} yields  for $r\ge 1$
	\begin{align}
		\Vert \xi^{\top} A\xi\Vert_{L_{r}}\lesssim \Vert \xi^{\top} A\tilde{\xi}\Vert_{L_{r}},\nonumber
	\end{align}
where $\tilde{\xi}=(\tilde{\xi}_{1}, \cdots, \tilde{\xi}_{n})$. Let $\eta_{1}, \cdots, \eta_{n}\stackrel{i.i.d.}{\sim}\mathcal{W}_{s}(\alpha), 0<\alpha\le 2$, and  denote by $\tilde{\eta}_{i}$ the  independent copy of $\eta_{i}$, $1\le i\le n$. Using the conditional probability and Lemma \ref{Lem_contraction_principle} twice, we have for $r\ge 1$
\begin{align}
	\Vert \xi^{\top} A\tilde{\xi}/L(\alpha)^{2}\Vert_{L_{r}}&=  \Big(\mathbb{E}_{\xi}\mathbb{E}_{\tilde{\xi}}\Big\vert \frac{1}{L(\alpha)^{2}}\sum_{j}\tilde{\xi}_{j}\sum_{i}a_{ij}\xi_{i}\Big\vert^{r}\Big)^{1/r}\nonumber\\
    &\lesssim  \Big(\mathbb{E}_{\xi}\mathbb{E}_{\tilde{\delta},\tilde{\eta}}\Big\vert \frac{1}{L(\alpha)}\sum_{j}\tilde{\delta}_{j}\tilde{\eta}_{j}\sum_{i}a_{ij}\xi_{i}\Big\vert^{r}\Big)^{1/r}\lesssim\Big\Vert \sum_{i,j}a_{ij}\delta_{i}\eta_{i}\tilde{\delta}_{j}\tilde{\eta}_{j}\Big\Vert_{L_{r}}.\nonumber
\end{align}
Corollary \ref{Cor_Section2_quadratic_forms} yields for $r\ge 2$
\begin{align}\label{Eq_Section3_Proposition2_1}
	\mathbb{E}_{\eta, \tilde{\eta}}\big\vert\sum_{i,j}a_{ij}\delta_{i}\eta_{i}\tilde{\delta}_{j}\tilde{\eta}_{j}   \big\vert^{r}\le& C_{1}(\alpha)^{r}r^{r/2}\Big(\sum_{i,j}a_{ij}^{2}\delta_{i}\tilde{\delta}_{j} \Big)^{r/2}\nonumber\\
	&+C_{2}(\alpha)^{r}r^{2r/\alpha}\Vert (a_{ij}\delta_{i}\tilde{\delta}_{j})_{n\times n}\Vert_{l_{2\to l_{2}}}^{r},
\end{align}
where $\mathbb{E}_{\eta, \tilde{\eta}}$ means taking expectation with respect to $\eta_{i}, \tilde{\eta}_{i}, 1\le i\le n$.


Let $\Lambda=\textnormal{Diag}(\delta_{1}, \cdots, \delta_{n})$ be a diagonal matrix with entries $\delta_{1}, \cdots, \delta_{n}$, and $\tilde{\Lambda}=\textnormal{Diag}(\tilde{\delta}_{1}, \cdots, \tilde{\delta}_{n})$. As for the second term on the right side of \eqref{Eq_Section3_Proposition2_1}, we have 
\begin{align}\label{Eq_Section3_prop2_2}
     \Vert (a_{ij}\delta_{i}\tilde{\delta}_{j})_{n\times n}\Vert_{l_{2\to l_{2}}}=\Vert \Lambda A\tilde{\Lambda}\Vert_{l_{2}\to l_{2}}
     \le \Vert \Lambda\Vert_{l_{2}\to l_{2}}\Vert  A\Vert_{l_{2}\to l_{2}}\Vert \tilde{\Lambda}\Vert_{l_{2}\to l_{2}}\le\Vert A\Vert_{l_{2}\to l_{2}}.
\end{align}

We next turn to bounding the first term of \eqref{Eq_Section3_Proposition2_1}. Note that,  $f(x_{1}, \cdots, x_{n})=\sqrt{\sum_{i}a_{i}^{2}x_{i}^{2}}$ is a convex and Lipschitz function with $\Vert f\Vert_{\textnormal{Lip}}=\max_{i}\vert a_{i}\vert$. Hence, Lemma \ref{Lem_concentration_inequality} yields for $t>0$
\begin{align}
	\mathbb{P}_{\delta}\Big\{ \Big\vert \big(\sum_{i}\delta_{i}\sum_{j}a_{ij}^{2}\tilde{\delta}_{j}\big)^{1/2}-\mathbb{E}_{\delta}\big(\sum_{i}\delta_{i}\sum_{j}a_{ij}^{2}\tilde{\delta}_{j}\big)^{1/2} \Big\vert >\max_{i}(\sum_{j}a_{ij}^{2}\tilde{\delta}_{j})^{1/2}t \Big\}\le 4e^{-t^{2}/4},\nonumber
\end{align}
where $\mathbb{P}_{\delta}$ and $\mathbb{E}_{\delta}$ mean taking probability and expectation with respect to $\delta_{1}, \cdots, \delta_{n}$. Then, a direct integration yields for $r\ge 1$ 
\begin{align}\label{Eq_Section3_prop_3}
	&\Big(\mathbb{E}_{\delta}\big\vert \big(\sum_{i}\delta_{i}\sum_{j}a_{ij}^{2}\tilde{\delta}_{j}\big)^{1/2} \big\vert^{r}\Big)^{1/r}\nonumber\\
\lesssim&\mathbb{E}_{\delta}\big(\sum_{i}\delta_{i}\sum_{j}a_{ij}^{2}\tilde{\delta}_{j}\big)^{1/2}+\sqrt{r}\max_{i}(\sum_{j}a_{ij}^{2}\tilde{\delta}_{j})^{1/2}.\nonumber\\
	\le&  \big(\sum_{i,j}a_{ij}^{2}p_{i}\tilde{\delta}_{j}\big)^{1/2}+\sqrt{r}\max_{i}(\sum_{j}a_{ij}^{2})^{1/2}.
\end{align}
It is obvious to see that
\begin{align}
	\max_{i}(\sum_{j}a_{ij}^{2})^{1/2}=\Vert A\Vert_{l_{2}\to l_{\infty}}\le \Vert A\Vert_{l_{2}\to l_{2}}.\nonumber
\end{align}
Using Lemma \ref{Lem_concentration_inequality} again, we have for $t>0$
\begin{align}
	\mathbb{P}\Big\{\Big\vert \big(\sum_{i,j}a_{ij}^{2}p_{i}\tilde{\delta}_{j}\big)^{1/2}-\mathbb{E}\big(\sum_{i,j}a_{ij}^{2}p_{i}\tilde{\delta}_{j}\big)^{1/2}\Big\vert >\max_{j}(\sum_{i}a_{ij}^{2}p_{i})^{1/2}t \Big\}\le 4e^{-t^{2}/4}.\nonumber
\end{align}
Then for $r\ge 1$
\begin{align}\label{Eq_Section3_prop_4}
	\big\Vert \big(\sum_{i,j}a_{ij}^{2}p_{i}\tilde{\delta}_{j}\big)^{1/2} \big\Vert_{L_{r}}&\lesssim \mathbb{E}\big(\sum_{i,j}a_{ij}^{2}p_{i}\tilde{\delta}_{j}\big)^{1/2}+\sqrt{r}\max_{j}(\sum_{i}a_{ij}^{2}p_{i})^{1/2}\nonumber\\
	&\le \big(\sum_{i,j}a_{ij}^{2}p_{i}p_{j}\big)^{1/2}+\sqrt{r}\Vert A\Vert_{l_{2}\to l_{2}}.
\end{align}
By virtue of \eqref{Eq_Section3_prop_3} and \eqref{Eq_Section3_prop_4}, we have for $r\ge 1$
\begin{align}\label{Eq_Section3_prop2_5}
	\sqrt{r}\big\Vert \big(\sum_{i,j}a_{ij}^{2}\delta_{i}\tilde{\delta}_{j}\big)^{1/2} \big\Vert_{L_{r}}\lesssim \sqrt{r}\big(\sum_{i,j}a_{ij}^{2}p_{i}p_{j}\big)^{1/2}+r\Vert A\Vert_{l_{2}\to l_{2}}.
\end{align}
As $0< \alpha\leq 2$, we can absorb the last term above with $r^{2/\alpha}\Vert A\Vert_{l_{2}\to l_{2}}$. Combining \eqref{Eq_Section3_Proposition2_1} with \eqref{Eq_Section3_prop2_5}, we have for $r\ge 2$ 
\begin{align}
	\big\Vert \sum_{ij}a_{ij}\delta_{i}\eta_{i}\tilde{\delta}_{j}\tilde{\eta}_{j}  \big\Vert_{L_{r}}\lesssim_{\alpha}\sqrt{r}\big(\sum_{ij}a_{ij}^{2}p_{i}p_{j}\big)^{1/2}+r^{2/\alpha}\Vert A\Vert_{l_{2}\to l_{2}}.\nonumber
\end{align}
Hence, we finish the proof by Lemma \ref{Lem_Moments_to_tailbound} and adjusting the universal constant.
\end{proof}

Now, we are prepared to prove Theorem \ref{Theo_main1}. 
\begin{proof}[Proof of Theorem \ref{Theo_main1}]
	By the triangle inequality, we have for $r\ge 1$
	\begin{align}
		\big\Vert \xi^{\top}A\xi-\mathbb{E}\xi^{\top}A\xi \big\Vert_{L_{r}}\le 	\big\Vert \sum_{i\neq j}a_{ij}\xi_{i}\xi_{j} \big\Vert_{L_{r}}+\big\Vert \sum_{i}a_{ii}(\xi_{i}^{2}-\mathbb{E}\xi_{i}^{2})\big\Vert_{L_{r}}.\nonumber
	\end{align}

Proposition \ref{Prop_Section3_diagonal_free} yields that
\begin{align}\label{Eq_Section_Proof1_1}
		\big\Vert \sum_{i\neq j}a_{ij}\xi_{i}\xi_{j}/L(\alpha)^{2} \big\Vert_{L_{r}}\lesssim_{\alpha}\sqrt{r}\big(\sum_{i\neq j}a_{ij}^{2}p_{i}p_{j} \big)^{1/2}+r^{2/\alpha}\Vert A\Vert_{l_{2}\to l_{2}}. 
\end{align}

Let $\tilde{\xi}_{i}$ be an independent copy of $\xi_{i}$, $1\le i\le n$. Note that
\begin{align}
	\big\Vert \sum_{i}a_{ii}(\xi_{i}^{2}-\mathbb{E}\xi_{i}^{2})\big\Vert_{L_{r}}=\big\Vert \sum_{i}a_{ii}(\xi_{i}^{2}-\mathbb{E}_{\tilde{\xi}_{i}}\tilde{\xi}_{i}^{2})\big\Vert_{L_{r}}\le 2\big\Vert \sum_{i}a_{ii}\varepsilon_{i}\delta_{i}\zeta_{i}^{2}\big\Vert_{L_{r}},\nonumber
\end{align}
where $\varepsilon_{1}, \cdots, \varepsilon_{n}$ is a sequence of i.i.d. Rademacher random variables. Due to that
\begin{align}
	\mathbb{P}\{\vert \varepsilon_{i}\zeta_{i}^{2}\vert >L(\alpha)^{2}t  \}\le \mathbb{P}\{\vert \zeta_{i}\vert >L(\alpha)\sqrt{t}  \}\le 2e^{-ct^{\alpha/2}},\nonumber
\end{align}
Lemma \ref{Prop_Section3_Bernstein_sparse} yields that
\begin{align}\label{Eq_Section3_Proof1_2}
	\big\Vert \sum_{i}a_{ii}(\xi_{i}^{2}-\mathbb{E}\xi_{i}^{2})/L(\alpha)^{2}\big\Vert_{L_{r}}\lesssim_{\alpha}\sqrt{r}\big(\sum_{i}a_{ii}^{2}p_{i} \big)+r^{2/\alpha}\max_{i}\vert a_{ii}\vert.
\end{align}

By virtue of \eqref{Eq_Section_Proof1_1} and \eqref{Eq_Section3_Proof1_2}, we have
\begin{align}
   &\big\Vert (\xi^{\top}A\xi-\mathbb{E}\xi^{\top}A\xi)/L(\alpha)^{2} \big\Vert_{L_{r}}\nonumber\\
   \lesssim_{\alpha}& \sqrt{r}\big(\sum_{i\neq j}a_{ij}^{2}p_{i}p_{j}+\sum_{i}a_{ii}^{2}p_{i} \big)^{1/2}+r^{2/\alpha}\Vert A\Vert_{l_{2}\to l_{2}},
\end{align}
where we use the fact $\max_{i}\vert a_{ii}\vert\le \Vert A\Vert_{l_{2}\to l_{2}}$.
Hence, the desired result follows from Lemma \ref{Lem_Moments_to_tailbound}. 
\end{proof}

\subsection{An improved sparse Hanson-Wright inequality for $0<\alpha\le 1$}	
We complete the proof of Theorem \ref{Theo_main2} in this subsection.

\begin{proof}[Proof of Theorem \ref{Theo_main2}]

We first assume $a_{ii}=0, 1\le i\le n$.
By Lemmas \ref{Lem_decoupling_inequality} and \ref{Lem_contraction_principle}, we only need to bound the quantity $\Vert \sum_{i, j}a_{ij}\delta_{i}\eta_{i}\tilde{\delta}_{j}\tilde{\eta}_{j}\Vert_{L_{r}}$. Here, $\eta_{1}, \cdots, \eta_{n}\stackrel{i.i.d.}{\sim}\mathcal{W}_{s}(\alpha)$, and for $1\le i\le n$, $\tilde{\delta}_{i}, \tilde{\eta}_{i}$ are independent copies of $\delta_{i}, \eta_{i}$, respectively.

Note that, for $1\le i\le n$,
$$		-\log \mathbb{P}\{\vert\delta_{i}\eta_{i}\vert\ge t \} = 
		\begin{cases}
			0, & t = 0 \\
			t^{\alpha}-\log p_{i}, & t > 0
		\end{cases}
$$
is a concave function for $t\ge 0$. Hence, by Lemma \ref{Lem_Lp_linearsum}, we have for $r\ge 2$
\begin{align}\label{Eq_Section3_proof2_1}
	\Big(\mathbb{E}_{\delta, \eta}\big\vert \sum_{i}\delta_{i}\eta_{i}\sum_{j}a_{ij}\tilde{\delta}_{j}\tilde{\eta}_{j} \big\vert^{r}\Big)^{1/r}\lesssim_{\alpha}& \,r^{1/\alpha}\Big( \sum_{i}p_{i}\big\vert \sum_{j}a_{ij}\tilde{\delta}_{j}\tilde{\eta}_{j} \big\vert^{r} \Big)^{1/r}\nonumber\\
	&+\sqrt{r}\Big( \sum_{i}p_{i}(\sum_{j}a_{ij}\tilde{\delta}_{j}\tilde{\eta}_{j})^{2}  \Big)^{1/2},
\end{align}
where $\mathbb{E}_{\delta, \eta}$ means taking expectation with respect to $\{\delta_{i}, \eta_{i}, 1\le i\le n  \}$.

Combining Fubini's theorem with \eqref{Eq_Section3_proof2_1}, we have for $r\ge 2$
\begin{align}\label{Eq_Section3_proof2_added_1}
    \Big\Vert \sum_{i, j}a_{ij}\delta_{i}\eta_{i}\tilde{\delta}_{j}\tilde{\eta}_{j}\Big\Vert_{L_{r}}=&\Big(\mathbb{E}_{\tilde{\delta}, \tilde{\eta}}\mathbb{E}_{\delta, \eta}\big\vert \sum_{i}\delta_{i}\eta_{i}\sum_{j}a_{ij}\tilde{\delta}_{j}\tilde{\eta}_{j} \big\vert^{r}\Big)^{1/r}\nonumber\\
    \lesssim_{\alpha} &r^{1/\alpha}\Big(\mathbb{E}\sum_{i}p_{i}\big\vert \sum_{j}a_{ij}\tilde{\delta}_{j}\tilde{\eta}_{j} \big\vert^{r}\Big)^{1/r} \nonumber\\
    &+\sqrt{r}\Big\Vert\Big( \sum_{i}p_{i}(\sum_{j}a_{ij}\tilde{\delta}_{j}\tilde{\eta}_{j})^{2}  \Big)^{1/2}\Big\Vert_{L_{r}},
\end{align}
where the last inequality we use the fact $(a+b)^{r}\le 2^{r}(\vert a\vert^{r}+\vert b\vert^{r})$.

Next, we turn to the terms on the right side of \eqref{Eq_Section3_proof2_added_1}.
For the first term, we have
\begin{align}\label{Eq_Section3_proof2_2}
	\Big(\mathbb{E}\sum_{i}p_{i}\big\vert \sum_{j}a_{ij}\tilde{\delta}_{j}\tilde{\eta}_{j} \big\vert^{r}\Big)^{1/r} &=\Big(\sum_{i}p_{i}\Big\Vert \sum_{j}a_{ij}\tilde{\delta}_{j}\tilde{\eta}_{j} \Big\Vert^{r}_{L_{r}}\Big)^{1/r}\nonumber\\
    &\lesssim \Big(\sum_{i}p_{i} \Big(\big(\sum_{j}\mathbb{E}\vert a_{ij}\tilde{\delta}_{j}\tilde{\eta}_{j}\vert^{r}\big)^{1/r}+\sqrt{r}\big(\sum_{j}\mathbb{E}\vert a_{ij}\tilde{\delta}_{j}\tilde{\eta}_{j}\vert^{2} \big)^{1/2}\Big)^{r} \Big)^{1/r}    \nonumber\\
    &\lesssim_{\alpha} r^{1/\alpha}\Big( \sum_{i,j}\vert a_{ij}\vert^{r} p_{i}p_{j}  \Big)^{1/r}
	+\sqrt{r}\Big(\sum_{i}p_{i}\big(\sum_{j}a_{ij}^{2}p_{j} \big)^{r/2}  \Big)^{1/r},
\end{align}
where we use Lemma \ref{Lem_Lp_linearsum} in the second step and use the triangle inequality in the third step.

 For the second term on the right side of \eqref{Eq_Section3_proof2_added_1}, note that for $g_{1}, \cdots, g_{n}\stackrel{i.i.d.}{\sim}N(0, 1)$ and $a_{1}, \cdots, a_{n}\in\mathbb{R}$
\begin{align}
	\Big\Vert \sum_{i}a_{i}g_{i}\Big\Vert_{L_{r}}\asymp\sqrt{r}\big(\sum_{i}a_{i}^{2}\big)^{1/2}.\nonumber
\end{align}
Hence, we have by Lemma \ref{Lem_Lp_linearsum}
\begin{align}
	&\mathbb{E}\Big(r \sum_{i}p_{i}(\sum_{j}a_{ij}\tilde{\delta}_{j}\tilde{\eta}_{j})^{2}  \Big)^{r/2}\nonumber\\
    \le& \,C^{r}\mathbb{E}\mathbb{E}_{g}\Big(\sum_{i}\sqrt{p_{i}}g_{i}\sum_{j}a_{ij}\tilde{\delta}_{j}\tilde{\eta}_{j}   \Big)^{r}\nonumber\\
	=& \, C^{r}\mathbb{E}\mathbb{E}_{\tilde{\delta}, \tilde{\eta}}\Big(\sum_{j}\tilde{\delta}_{j}\tilde{\eta}_{j} \sum_{i}a_{ij}\sqrt{p_{i}}g_{i}  \Big)^{r}\nonumber\\
    \le& C_{1}(\alpha)^{r} \mathbb{E}\Big( r^{r/\alpha}\sum_{j}p_{j}\big\vert \sum_{i}a_{ij}\sqrt{p_{i}}g_{i}\big\vert^{r}+r^{r/2}\big(\sum_{j}p_{j}\big(\sum_{i}a_{ij}\sqrt{p_{i}}g_{i}\big)^{2} \big)^{r/2}   \Big)\nonumber\\
	\le & \, C_{2}(\alpha)^{r}r^{r/\alpha+r/2}\sum_{j}p_{j}\big(\sum_{i}a_{ij}^{2}p_{i} \big)^{r/2}\, +C_{3}(\alpha)^{r}\mathbb{E}\Big(\sum_{i, j}a_{ij}\sqrt{p_{i}p_{j}}g_{i}\tilde{g}_{j} \Big)^{r},\nonumber
\end{align}
where $\tilde{g}_{j}$ is an independent copy of $g_{j}$, $1\le j\le n$.
Note that $g_{i}$ are also sub-gaussian random variables.
Hence, we have by Corollary \ref{Cor_Section2_quadratic_forms} and Lemma \ref{Lem_contraction_principle}
\begin{align}
	\Big\Vert \sum_{i, j}a_{ij}\sqrt{p_{i}p_{j}}g_{i}\tilde{g}_{j}\Big\Vert_{L_{r}}\lesssim \sqrt{r}\Big(\sum_{i,j}a_{ij}^{2}p_{i}p_{j} \Big)^{1/2}+r\Vert (a_{ij}\sqrt{p_{i}p_{j}})_{n\times n}\Vert_{l_{2}\to l_{2}}.\nonumber
\end{align}

Up to now, we have proved for diagonal free symmetric matrix $A$
\begin{align}
	\Big\Vert  \sum_{i,j}a_{ij}\delta_{i}\eta_{i}\tilde{\delta}_{j}\tilde{\eta}_{j} \Big\Vert_{L_{r}}\lesssim_{\alpha}\,& r^{2/\alpha}\Big(\sum_{ij}\vert a_{ij}\vert^{r}p_{i}p_{j}  \Big)^{1/r} + r^{1/2+1/\alpha}\Big(\sum_{i}p_{i}\big(\sum_{j}a_{ij}^{2}p_{j} \big)^{r/2}  \Big)^{1/r}\nonumber\\
	&+r\Vert (a_{ij}\sqrt{p_{i}p_{j}})_{n\times n}\Vert_{l_{2}\to l_{2}}
    +\sqrt{r}\Big(\sum_{i,j}a_{ij}^{2}p_{i}p_{j} \Big)^{1/2}.\nonumber
\end{align}
Note that for $r\ge 3$, by Young's inequality, we have
\begin{align}
	\Big(\sum_{i,j}\vert a_{ij}\vert^{r}p_{i}p_{j} \Big)^{1/r}&\le \max_{i,j}\vert a_{ij}\vert^{(r-2)/r}\cdot\Big( \sum_{i, j}a_{ij}^{2}p_{i}p_{j} \Big)^{1/r}\nonumber\\
	&=\Big( e^{2r/(r-2)}\max_{i,j}\vert a_{ij}\vert \Big)^{(r-2)/r}\cdot\Big( \sum_{i ,j}a_{ij}^{2}p_{i}p_{j}e^{-2r} \Big)^{1/r}\nonumber\\
	&\le \frac{r-2}{r}\Big( e^{2r/(r-2)}\max_{i,j}\vert a_{ij}\vert \Big)+\frac{2}{r}\Big( \sum_{i, j}a_{ij}^{2}p_{i}p_{j}e^{-2r} \Big)^{1/2}\nonumber\\
	&\le e^{6}\Vert A\Vert_{\infty}+e^{-r}\Big( \sum_{i, j}a_{ij}^{2}p_{i}p_{j} \Big)^{1/2}.\nonumber
\end{align}
Similarly
\begin{align}
	\Big(\sum_{i}p_{i}\big(\sum_{j}a_{ij}^{2}p_{j} \big)^{r/2}  \Big)^{1/r}\le e^{6}\max_{i} \Big(\sum_{j}a_{ij}^{2}p_{j} \Big)^{1/2}+e^{-r}\Big( \sum_{i, j}a_{ij}^{2}p_{i}p_{j} \Big)^{1/2}.\nonumber
\end{align}
Note that the $e^{-r}$ factors we introduced will eliminate the $r^{2/\alpha}, r^{1/2 +1/\alpha}$ factors in the corresponding terms. Hence, we have for $r\ge 3$
\begin{align}\label{Eq_Section3_proof2_3}
	\Big\Vert  \sum_{i,j}a_{ij}\delta_{i}\eta_{i}\tilde{\delta}_{j}\tilde{\eta}_{j} \Big\Vert_{L_{r}}\lesssim_{\alpha}&r^{2/\alpha}\Vert A\Vert_{\infty}+r^{1/2+1/\alpha}\max_{i} \Big(\sum_{j}a_{ij}^{2}p_{j} \Big)^{1/2}\nonumber\\
	&+r\Vert (a_{ij}\sqrt{p_{i}p_{j}})_{n\times n}\Vert_{l_{2}\to l_{2}}+\sqrt{r}\Big(\sum_{i,j}a_{ij}^{2}p_{i}p_{j} \Big)^{1/2}.
\end{align}

For a general symmetric matrix $A$, we only need to estimate the diagonal case. Following the same line as in the proof of Theorem \ref{Theo_main1}, we have 
\begin{align}\label{Eq_Section3_proof2_4}
	\big\Vert \sum_{i}a_{ii}(\xi_{i}^{2}-\mathbb{E}\xi_{i}^{2})/L(\alpha)^{2}\big\Vert_{L_{r}}\lesssim_{\alpha}\sqrt{r}\big(\sum_{i}a_{ii}^{2}p_{i} \big)^{1/2}+r^{2/\alpha}\max_{i}\vert a_{ii}\vert.
\end{align}
Combining \eqref{Eq_Section3_proof2_3} with \eqref{Eq_Section3_proof2_4}, the desired result follows from Lemma \ref{Lem_Moments_to_tailbound}.
\end{proof}
\section{Proofs of applications}
\subsection{Proof of Theorem \ref{thm:RIP}}
Although Theorem~\ref{thm:RIP} is not a literal corollary of our main theorem, its proof follows the same roadmap we developed for Theorem~\ref{Theo_main1}; thus it can be viewed as an application of our techniques. Before prove Theorem \ref{thm:RIP}, we first show the following properties.

\begin{prop}\label{Proposition_application_covariace_1}

Let $\delta\in\mathbb{R}^d$ and $B\in \mathbb{R}^{d\times m}$  as in Theorem \ref{thm:RIP}. Let $A=\mathrm{Diag}(a_1, \cdots, a_d)$ be a diagonal matrix. For any $r\ge 1$, we have
\begin{align}
    &\big(\mathbb{E}\Vert  B^\T\cdot \mathrm{Diag}(\delta)\cdot A\cdot \mathrm{Diag}(\delta)\cdot B\Vert^{r}_{F} \big)^{1/r}\nonumber\\
    \lesssim& \Vert B\Vert_{l_2\to l_2}\cdot\Vert A\Vert_{l_2\to l_2}\big(\Vert \mathrm{Diag}(\sqrt{p_1}, \cdots, \sqrt{p_d} )\cdot B\Vert_{F}+ r^{1/2}\Vert B\Vert_{l_2\to l_2}\big).
\end{align}  
\end{prop}
\begin{proof}


    Note that
    \begin{align}
        &\Vert  B^\T\cdot \mathrm{Diag}(\delta)\cdot A\cdot \mathrm{Diag}(\delta)\cdot B\Vert_{F}\nonumber\\
        \le& \Vert B^\T \cdot\mathrm{Diag}(\delta)\cdot\sqrt{A}\Vert_{l_2\to l_2}\cdot\Vert \sqrt{A}\cdot\mathrm{Diag}(\delta)\cdot B\Vert_{F}\nonumber\\
        \le& \Vert B\Vert_{l_2\to l_2}\cdot\Vert \sqrt{A}\Vert_{l_2\to l_2}\cdot\Vert \sqrt{A}\cdot\mathrm{Diag}(\delta)\cdot B\Vert_{F},
    \end{align}
    where $\sqrt{A}$ is a (complex) diagonal matrix such that $\sqrt{A}\cdot\sqrt{A}=A$. Talagrand's concentration inequality (Lemma \ref{Lem_concentration_inequality}) yields that
    \begin{align}
        &\mathbb{P}\Big\{\big\vert \Vert \sqrt{A}\cdot\mathrm{Diag}(\delta)\cdot B\Vert_{F}-\mathbb{E}\Vert \sqrt{A}\cdot\mathrm{Diag}(\delta)\cdot B\Vert_{F}  \big\vert>t   \Big\}\nonumber\\
        \le& 2\exp\{-\frac{t^2}{\big(\Vert \sqrt{A}\Vert_{l_2\to l_2}\cdot \Vert B\Vert_{l_2\to \infty}\big)^2} \}.
    \end{align}
    We finish the proof due to that
    \begin{align}
        \mathbb{E}\Vert \sqrt{A}\cdot\mathrm{Diag}(\delta)\cdot B\Vert_{F}&\le \big(\mathbb{E}\Vert \sqrt{A}\cdot\mathrm{Diag}(\delta)\cdot B\Vert_{F}^2\big)^{1/2}\nonumber\\
        &\le \Vert \sqrt{A}\Vert_{l_2\to l_2}\cdot \Vert \mathrm{Diag}(\sqrt{p_1}, \cdots, \sqrt{p_d} )\cdot B\Vert_{F}.
    \end{align}
\end{proof}

\begin{prop}\label{Proposition_application_covariace_2}
    Let $\delta\in\mathbb{R}^d$ and $B\in \mathbb{R}^{d\times m}$  as in Theorem \ref{thm:RIP}. For a fixed vector $x\in\mathbb{R}^{d}$, denote $A=xx^\T$. For any $r\ge 1$, we have
\begin{align}
    \big(\mathbb{E}\Vert  B^\T\cdot \mathrm{Diag}(\delta)\cdot A\cdot \mathrm{Diag}(\delta)\cdot B\Vert^{r}_{F} \big)^{1/r}\lesssim \Vert B\Vert_{l_2\to l_2}^2\Vert x\Vert_{2}^{2}.
\end{align}  

\end{prop}

\begin{proof}
    

\begin{align}
    &\Vert  B^\T\cdot \mathrm{Diag}(\delta)\cdot A\cdot \mathrm{Diag}(\delta)\cdot B\Vert_{F}\nonumber\\
    =&\Vert  B^\T\cdot \mathrm{Diag}(x)\cdot \delta\delta^\T\cdot \mathrm{Diag}(x)\cdot B\Vert_{F}\nonumber\\
    =&\Vert  B^\T\cdot \mathrm{Diag}(x)\cdot \delta\Vert_{2}^{2}\le \Vert B\Vert_{l_2\to l_2}^2 \Vert \mathrm{Diag}(x)\cdot\delta\Vert^2_{2}\le \Vert B\Vert_{l_2\to l_2}^2\Vert x\Vert_{2}^{2}.
\end{align}

\end{proof}

\begin{proof}[Proof of Theorem~\ref{thm:RIP}]
    For each fixed $\theta$, we can rewrite
    \begin{align}
        \theta^\T(\widehat{\Sigma} - \Sigma)\theta& = \frac{1}{n}\sum_{s=1}^n \left[{X^{(s)}}^\T A_{\theta,p} X^{(s)} - \E({X^{(s)}}^\T A_{\theta,p} X^{(s)})\right]\nonumber\\
        &= \frac{1}{n}\sum_{s=1}^n \left[{\xi^{(s)}}^\T B^\T \mathrm{Diag}(\delta^{(s)})A_{\theta, p}\mathrm{Diag}(\delta^{(s)})B\xi^{(s)} - \E({X^{(s)}}^\T A_{\theta,p} X^{(s)})\right]\nonumber\\
        &=\tilde{\xi}^\T\tilde{B}^\T\mathrm{Diag}(\tilde{\delta})\tilde{A}_{\theta, p}\mathrm{Diag}(\tilde{\delta})\tilde{B}\tilde{\xi}-\mathbb{E}\tilde{\xi}^\T\tilde{B}^\T\mathrm{Diag}(\tilde{\delta})\tilde{A}_{\theta, p}\mathrm{Diag}(\tilde{\delta})\tilde{B}\tilde{\xi}.
    \end{align}
    Here
    \begin{align}
        &\tilde{\xi}=({\xi^{(1)}}^\T, \cdots,{\xi^{(n)}}^\T)^\T\in \mathbb{R}^{mn},\quad \tilde{B}=\mathrm{Diag}(B, B, \cdots, B)\in R^{dn\times mn},\nonumber\\
        &\tilde{\delta}=({\delta^{(1)}}^\T, \cdots, {\delta^{(n)}}^\T)^{\T}\in \mathbb{R}^{dn\times dn},\quad \tilde{A}_{\theta, p}=\mathrm{Diag}(A_{\theta, p}, \cdots, A_{\theta, p})/n\in\mathbb{R}^{dn\times dn}.
    \end{align}
Hence, by virtue of the decoupling inequality and Corollary \ref{Cor_Section2_quadratic_forms}, we have for $r\ge 1$
\begin{align}\label{Eq_bound1}
    \Vert \theta^\T(\widehat{\Sigma} - \Sigma)\theta\Vert_{L_r}&\lesssim_{\alpha}\Vert \tilde{\xi}^\T\tilde{B}^\T\mathrm{Diag}(\tilde{\delta})\tilde{A}_{\theta, p}\mathrm{Diag}(\tilde{\delta})\tilde{B}\tilde{\eta}\Vert_{L_r}\nonumber\\
    &\lesssim_{\alpha} r^{1/2}\big(\mathbb{E}\Vert  \tilde{B}^\T\mathrm{Diag}(\tilde{\delta})\tilde{A}_{\theta, p}\mathrm{Diag}(\tilde{\delta})\tilde{B}\Vert_{F}^{r}\big)^{1/r}\nonumber\\
    & \qquad+r^{2/\alpha}\big(\mathbb{E}\Vert  \tilde{B}^\T\mathrm{Diag}(\tilde{\delta})\tilde{A}_{\theta, p}\mathrm{Diag}(\tilde{\delta})\tilde{B}\Vert_{l_2\to l_2}^{r}\big)^{1/r},
    \end{align}
    where $\tilde{\eta}$ is an independent copy of $\tilde{\xi}$. 

For the first term on the right-side of \eqref{Eq_bound1}, we have
\begin{align}
    \Vert  \tilde{B}^\T\mathrm{Diag}(\tilde{\delta})\tilde{A}_{\theta, p}\mathrm{Diag}(\tilde{\delta})\tilde{B}\Vert_{F}=\Big(\frac{1}{n^2}\sum_{s=1}^{n}\Vert B^\T\mathrm{Diag}(\delta^{(s)})A_{\theta, p}\mathrm{Diag}(\delta^{(s)}) B\Vert_{F}^2\Big)^{1/2}
\end{align}
Note that
\begin{align}\label{Eq_123456}
    A_{\theta, p}=(\theta\circ\frac{1}{p})(\theta\circ\frac{1}{p})^\T-\mathrm{Diag}((\theta\circ\frac{1}{p})(\theta\circ\frac{1}{p})^\T)+\mathrm{Diag}(A_{\theta, p}).
\end{align}
Hence, we can split the above boundary into two parts according to the division of $A_{\theta, p}$, the boundary formed by $(\theta\circ\frac{1}{p})(\theta\circ\frac{1}{p})^\T$ and the boundary formed by $\mathrm{Diag}(A_{\theta, p})-\mathrm{Diag}((\theta\circ\frac{1}{p})(\theta\circ\frac{1}{p})^\T)$. Propositions \ref{Proposition_application_covariace_1} and \ref{Proposition_application_covariace_2} yield that, $\Vert B^\T\mathrm{Diag}(\delta)A_{\theta, p}\mathrm{Diag}(\delta)B\Vert_{F}$ is a sub-gaussian random variable with  sub-gaussian norm 
\begin{align}
    K_1(\theta)&=\Vert B\Vert_{l_2\to l_2}\cdot(\mathrm{Diag}(\sqrt{p_1}, \cdots, \sqrt{p_d} )\cdot B\Vert_{F}+\Vert B\Vert_{l_2\to l_2})\cdot\Vert(\theta\circ\frac{1}{p})(\theta\circ\frac{1}{p})^\T \Vert_{l_2\to l_2}\nonumber\\
    &=\Vert B\Vert_{l_2\to l_2}\cdot(\mathrm{Diag}(\sqrt{p_1}, \cdots, \sqrt{p_d} )\cdot B\Vert_{F}+\Vert B\Vert_{l_2\to l_2})\cdot\Vert\theta\circ\frac{1}{p} \Vert^2_{2}
\end{align}
Hence, Bernstein's inequality yields that
\begin{align}
    \mathbb{P}\Big\{\Big\vert\frac{1}{n^2}\sum_{s=1}^{n}\Big(\Vert B^\T\mathrm{Diag}(\delta^{(s)})A_{\theta, p}\mathrm{Diag}(\delta^{(s)})B\Vert_{F}^{2}-&\mathbb{E}\Vert B^\T\mathrm{Diag}(\delta^{(s)})A_{\theta, p}\mathrm{Diag}(\delta^{(s)})B\Vert_{F}^2\Big) \Big\vert\nonumber\\
    &>K_{1}(\theta)^2t \Big\}
    \le 2\exp\{-c\min\{ n^3t^2, n^2t \}\},
\end{align}
implying that
\begin{align}
    \Big(\mathbb{E}\Big(\frac{1}{n^2}\sum_{s=1}^{n}\Vert B^\T\mathrm{Diag}(\delta^{(s)})A_{\theta, p}\mathrm{Diag}(\delta^{(s)}) B\Vert_{F}^2\Big)^{r/2} \Big)^{2/r}&\lesssim \frac{1}{n}\mathbb{E}\Vert B^\T\mathrm{Diag}(\delta^{(s)})A_{\theta, p}\mathrm{Diag}(\delta^{(s)})B\Vert_{F}^2\nonumber\\
    &+(\frac{r}{n^3})^{1/2}K^2_{1}(\theta)+\frac{r}{n^2}K^2_{1}(\theta)
\end{align}
Hence, by virtue of \eqref{Eq_123456}, the first term on the right hand side of \eqref{Eq_bound1} can be further bounded by 
\begin{align}
    (\frac{r}{n})^{1/2}K_2(\theta)+(\frac{r}{n})^{3/4}K_{1}(\theta)+\frac{r}{n}K_{1}(\theta),
\end{align}
where
\begin{align}
    K_2(\theta):=\Big(\mathbb{E}\Vert B^\T\mathrm{Diag}(\delta^{(1)})A_{\theta, p}\mathrm{Diag}(\delta^{(1)})B\Vert_{F}^2\Big)^{1/2}.
\end{align}

For the second term  on the right-side of \eqref{Eq_bound1}, we have
\begin{align}
    \Vert  \tilde{B}^\T\mathrm{Diag}(\tilde{\delta})\tilde{A}_{\theta, p}\mathrm{Diag}(\tilde{\delta})\tilde{B}\Vert_{l_2\to l_2}\le \frac{1}{n}\Vert B\Vert_{l_2\to l_2}^2\cdot\Vert A_{\theta, p}\Vert_{l_2\to l_2}.
\end{align}
Hence, we have for $r\ge1$
\begin{align}
    \Vert \theta^\T(\widehat{\Sigma} - \Sigma)\theta\Vert_{L_r}&\lesssim_{\alpha} \frac{K_2(\theta)r^{1/2}}{\sqrt{n}}+\frac{K_{1}(\theta)r^{3/4}}{n^{3/4}}+\frac{K_{1}(\theta)r^{2/\alpha}}{n},
\end{align}
implying that
\begin{align}\label{Eq_application_covariance_1}
   &\mathbb{P}\Big\{ \big\vert \theta^\T(\widehat{\Sigma} - \Sigma)\theta\big\vert>\frac{t^{1/2}}{\sqrt{n}}\sup_{\theta\in\Omega}K_2(\theta)+\frac{t^{3/4}}{n^{3/4}}\sup_{\theta\in\Omega}K_1(\theta)+\frac{t^{2/\alpha}}{n}\sup_{\theta\in\Omega}K_1(\theta)\Big\}\nonumber\\
   \le&\mathbb{P}\Big\{ \big\vert \theta^\T(\widehat{\Sigma} - \Sigma)\theta\big\vert>\frac{K_2(\theta)t^{1/2}}{\sqrt{n}}+\frac{K_{1}(\theta)t^{3/4}}{n^{3/4}}+\frac{K_{1}(\theta)t^{2/\alpha}}{n}\Big\}\le  e^{-c(\alpha)t},
\end{align}
where $c(\alpha)$ is a constant only depending on $\alpha$.

For the set $\Omega$, there exists an $\varepsilon$-net $\mathcal{N}$ of $\Omega
$ such that, for $\varepsilon\in (0, 1)$,
\begin{align}
    \vert \mathcal{N}\vert\le \binom{d}{k}(1+2/\varepsilon)^k\le (\frac{ed}{k})^k(\frac{3}{\varepsilon})^k\le (\frac{3ed}{\varepsilon k})^k.
\end{align}
Hence, for any $\theta\in\Omega$, there exists a $\theta^\prime\in \mathcal{N}$ satisfying
\begin{align}
    \big\vert \theta^\T(\widehat{\Sigma} - \Sigma)\theta\big\vert&\le \big\vert (\theta-\theta^\prime)^\T(\widehat{\Sigma} - \Sigma)\theta+\theta^{\prime\T}(\widehat{\Sigma} - \Sigma)(\theta-\theta^\prime)\big\vert+\big\vert \theta^{\prime\T}(\widehat{\Sigma} - \Sigma)\theta^\prime\big\vert\nonumber\\
    &\le 2\varepsilon\sup_{\tilde{\theta}\in \Omega}\big\vert \tilde{\theta}^\T(\widehat{\Sigma} - \Sigma)\theta\big\vert+\big\vert \theta^{\prime\T}(\widehat{\Sigma} - \Sigma)\theta^\prime\big\vert.\nonumber
\end{align}
Due to the symmetry of $\widehat{\Sigma} - \Sigma$, we have
\begin{align}
     \tilde{\theta}^\T(\widehat{\Sigma} - \Sigma)\theta=\frac{1}{4}\Big( (\tilde{\theta}+\theta)^\T(\widehat{\Sigma} - \Sigma)(\tilde{\theta}+\theta)-(\tilde{\theta}-\theta)^\T(\widehat{\Sigma} - \Sigma)(\tilde{\theta}-\theta)\Big).
\end{align}
Note that $\Vert\tilde{\theta}\pm \theta\Vert_2\le 2 $. Hence, we have
\begin{align}
    \sup_{\theta\in\Omega}\big\vert \theta^\T(\widehat{\Sigma} - \Sigma)\theta\big\vert\le 4\varepsilon \sup_{\theta\in\Omega}\big\vert \theta^\T(\widehat{\Sigma} - \Sigma)\theta\big\vert+ \sup_{\theta^\prime\in\mathcal{N}}\big\vert \theta^{\prime\T}(\widehat{\Sigma} - \Sigma)\theta^\prime\big\vert.
\end{align}
Taking $\varepsilon=1/16$, we have
\begin{align}
    &\mathbb{P}\Big\{\frac{3}{4} \sup_{\theta\in\Omega}\big\vert \theta^\T(\widehat{\Sigma} - \Sigma)\theta\big\vert>\frac{t^{1/2}}{\sqrt{n}}\sup_{\theta\in\Omega}K_2(\theta)+\frac{t^{3/4}}{n^{3/4}}\sup_{\theta\in\Omega}K_1(\theta)+\frac{t^{2/\alpha}}{n}\sup_{\theta\in\Omega}K_1(\theta)\Big\}\nonumber\\
    \le& \mathbb{P}\Big\{ \sup_{\theta^\prime\in\mathcal{N}}\big\vert \theta^{\prime\T}(\widehat{\Sigma} - \Sigma)\theta^\prime\big\vert>\frac{t^{1/2}}{\sqrt{n}}\sup_{\theta\in\Omega}K_2(\theta)+\frac{t^{3/4}}{n^{3/4}}\sup_{\theta\in\Omega}K_1(\theta)+\frac{t^{2/\alpha}}{n}\sup_{\theta\in\Omega}K_1(\theta)\Big\}\nonumber\\
    \le &\vert\mathcal{N}\vert 2e^{-c(\alpha)t}\le 2\exp(-c(\alpha)t+k\log(48ed/k))
    \end{align}
We conclude the proof by taking $t=(u+k\log(48ed/k))/c(\alpha)$.

\end{proof}

\subsection{Proof of Theorem \ref{Theo_application_low}}
\begin{proof}[Proof of Theorem \ref{Theo_application_low}]
    Let $Q=[Q_1, \cdots, Q_r]\in\mathbb{R}^{k\times r}$ be a random matrix with independent copies of $r^{-\frac{1}{2}}\delta\circ\xi$, where $\delta\sim \mathrm{Bernoulli}(p)$ and $\xi$ is a centered sub-gaussian random variable satisfying $\Vert \xi\Vert_{\Psi_2}\le 1$ and $\mathbb{E}\xi^2=1$. The random variable $\delta$ is independent of $\xi$ and $0\le p\le 1$.
    
    Let $X=U\Sigma V^\T$ be the thin singular value decomposition of $X$, where $U\in\mathbb{R}^{m\times k}, \Sigma\in \mathbb{R}^{k\times k}$ and $V\in \mathbb{R}^{n\times k}$. Define the following matrices
    \begin{align}
        U\Sigma^{1/2}=\widetilde{U}=[\tilde{u}^\T_1, \cdots, \tilde{u}^\T_m]^\T,\quad V\Sigma^{1/2}=\widetilde{V}=[\tilde{v}^\top_1, \cdots, \tilde{v}^\top_n]^\top.
    \end{align}
    Then we have
    \begin{align}
        \tilde{u}_i\tilde{v}^\top_{j}=X_{ij},\quad \mathbb{E}\tilde{u}_i Q Q^\T \tilde{v}^\T_j= p\cdot X_{ij},\nonumber
    \end{align}
    where $1\le i\le m$ and $1\le j\le n$. We rewrite $\tilde{u}_i Q Q^\T \tilde{v}^\T_j$ as a quadratic form
    \begin{align}
        \tilde{u}_i Q Q^\T \tilde{v}^\T_j=\mathrm{vec}(Q)^\T\cdot (I_{r}\otimes \tilde{u}_i^\T \tilde{v}_j)\cdot\mathrm{vec}(Q),
    \end{align}
    where $\mathrm{vec}(Q)=[Q^\T_1,\cdots, Q^\T_r]^\T\in \mathbb{R}^{kr}$ and $\otimes$ is the Kronecker product. Then, we have by Theorem \ref{Theo_main1}
    \begin{align}
        \mathbb{P}\Big\{\big\vert\tilde{u}_i Q Q^\T \tilde{v}^\T_j -p\cdot X_{ij}\big\vert>t  \Big\}&\le 2\exp\big(-c\min\{\frac{r^2t^2}{p\Vert I_{r}\otimes \tilde{u}_i^\T \tilde{v}_j\Vert_{F}^2 }, \frac{rt}{\Vert I_{r}\otimes \tilde{u}_i^\T \tilde{v}_j\Vert_{l_2\to l_2}} \}\big)\nonumber\\
        &\le 2\exp\big(-c\min\{\frac{rt^2}{p\Vert \tilde{u}_i\Vert_{2}^2\Vert \tilde{v}_j\Vert_{2}^2 }, \frac{rt}{\Vert \tilde{u}_i\Vert_{2}\Vert \tilde{v}_j\Vert_{2}} \}\big).
    \end{align}
    To obtain the first inequality, we use the following fact: $$p^2 \sum_{i\neq j} a_{ij}^2 + p\sum_i a_{ii}^2 \leq p\sum_{i,j}a_{ij}^2.$$ Hence, we have for $t>0$
    \begin{align}
        \mathbb{P}\Big\{\big\vert\tilde{u}_i Q Q^\T \tilde{v}^\T_j -p\cdot X_{ij}\big\vert>\Vert \tilde{u}_i\Vert_{2}\Vert \tilde{v}_j\Vert_{2}t  \Big\}\le \exp(-cr\min\{\frac{t^2}{p}, t \}).
    \end{align}
    Then by the union bound, we have 
    \begin{align}
        \mathbb{P}\Big\{\max_{i, j}\big\vert\tilde{u}_i Q Q^\T \tilde{v}^\T_j -p\cdot X_{ij}\big\vert>&\frac{kt}{\sqrt{mn}}\sqrt{\mu_\mathrm{col}\mu_\mathrm{row}}\cdot\Vert X\Vert_{l_2\to l_2}  \Big\}\\
        \le& mn\exp(-cr\min\{\frac{t^2}{p}, t \}).
    \end{align}
    Assume that $$r>\log \frac{mn}{\eta}\cdot\max\{\frac{p}{ct^2}, \frac{1}{ct}\},$$ then with probability at least $1-\eta$,  for $1\le i\le m$ and $1\le j\le n$
    \begin{align}
        \big\vert X_{ij}-\frac{1}{p}\tilde{u}_i Q Q^\T \tilde{v}^\T_j\big\vert \le\frac{kt}{p\sqrt{mn}}\sqrt{\mu_\mathrm{col}\mu_\mathrm{row}}\cdot\Vert X\Vert_{l_2\to l_2}.\nonumber
    \end{align}
    The desired result follows by setting $Y=(Y_{ij})$ such that $Y_{ij}=\frac{1}{p}\tilde{u}_i Q Q^\T \tilde{v}^\T_j$. 
\end{proof}

\paragraph{Acknowledgment}K. Wang was partially supported by Hong Kong RGC grant GRF 16304222. Y. Zhu was partially supported by an AMS-Simons Travel Grant and the Simons Grant MPS-TSM-00013944. This work was partially completed during the first author's visit to Professor Hanchao Wang at Shandong University. The first author would like to express sincere gratitude to Professor Wang for his valuable discussions and support during the visit, which significantly contributed to the development of the ideas presented in this paper.


\bibliographystyle{plain}
\bibliography{NewHW}


	

	
	
	



\appendix

\section{Proof of Proposition \ref{Lem_appendix_1}}\label{appendix_A}
\begin{proof}
	Let $\tilde{\eta}=(\tilde{\eta}_{1}, \cdots, \tilde{\eta}_{n})^{\top}$ be an independent copy of $\eta$. Lemma \ref{Lem_contraction_principle} and Remark \ref{Rem_contraction_principle} yield for $r\ge 1$
	\begin{align}
		\Vert \eta^\top A\eta-\mathbb{E}\eta^\top A\eta\Vert_{L_{r}}\asymp \Vert \eta^\top A\tilde{\eta}\Vert_{L_{r}}.\nonumber
	\end{align}
By virtue of Lemma \ref{Lem_Lp_quadratic_0_1}, we have for $r\ge 2$
\begin{align}
	\Vert \eta^\top A\tilde{\eta}\Vert_{L_{r}}\asymp_{\alpha}& \sqrt{r}\Vert A\Vert_{F}+r\Vert A\Vert_{l_{2}\to l_{2}}\nonumber\\
	&+r^{1/\alpha+1/2}\max_{i}\big(\sum_{j}a_{ij}^{2} \big)^{1/2}+r^{2/\alpha}\Vert A\Vert_{\infty}:=f_{1}(A, r).\nonumber
\end{align}
Hence, for $r\ge 2$, we have by the Paley-Zygmund inequality
\begin{align}
	\mathbb{P}\Big\{ \vert \eta^\top A\eta-\mathbb{E}\eta^\top A\eta\vert\ge\frac{C_{1}(\alpha)}{2} f_{1}(A, r)\Big\}&\ge \mathbb{P}\Big\{ \vert \eta^\top A\tilde{\eta}\vert\ge \frac{1}{2}\Vert \eta^\top A\tilde{\eta}\Vert_{L_{r}}\Big\}\nonumber\\
	&\ge (1-2^{-r})^{2}\Big( \frac{\Vert \eta^\top A\tilde{\eta}\Vert_{L_{r}}}{\Vert \eta^\top A\tilde{\eta}\Vert_{L_{2r}}}   \Big)^{2r}\ge \frac{1}{2}e^{-c_1(\alpha)r},\nonumber
\end{align}
where in the last inequality, we use Lemma~\ref{Lem_Lp_quadratic_0_1}.
If we lower bound $e^{-c_1(\alpha)r}$ by $e^{-c_1(\alpha)(r+2)}$, this inequality is valid for all $r\ge 0$. Hence, we have for $t\ge 0$
\begin{align}
		\mathbb{P}\Big\{ \vert \eta^\top A\eta-\mathbb{E}\eta^\top A\eta\vert\ge\frac{C_{1}(\alpha)}{2} f_{1}(A, t)\Big\}\ge \frac{1}{2}e^{-2c_{1}(\alpha)}e^{-c_{1}(\alpha)t},\nonumber
\end{align}
which immediately yields the desired result.
\end{proof}

\section{Proof of Theorem \ref{Theorem_application}}\label{appendix_application}
\begin{proof}[Proof of Theorem \ref{Theorem_application}]
    Without loss of generality, we assume $L(\alpha)=1$. Note that 
$
        \mathbb{E}\Vert A\xi\Vert_{2}^{2}=p\Vert A\Vert_{F}^2.
$  
Hence, Theorem \ref{Theo_main1} yields for $u>0$
\begin{align}\label{Eq_application_1}
    &\mathbb{P}\Big\{\Big\vert\Vert A\xi\Vert_{2}^{2}-p\Vert A\Vert_{F}^2 \Big\vert >t\Big\}\nonumber\\
    \le& 2\exp\Big\{-c(\alpha)\min\{\frac{t^2}{p\Vert A^\top A\Vert_{F}^2},\big(\frac{t}{\Vert A^\top A\Vert_{l_2\to l_2}}\big)^{\alpha/2} \} \Big\}.
\end{align}

Take $t=up\Vert A\Vert_{F}^2$, $u\ge 0$, and note that
\begin{align}
    \Vert A^\top A\Vert_{F}\le \Vert A\Vert_{l_2\to l_2} \Vert A\Vert_F,\quad \Vert A^\top A\Vert_{l_2\to l_2}=\Vert A\Vert_{l_2\to l_2}^2.
\end{align}
Then, \eqref{Eq_application_1} yields that
\begin{align}\label{Eq_application2}
    &\mathbb{P}\Big\{\Big\vert\Vert A\xi\Vert_{2}^{2}-p\Vert A\Vert_{F}^2 \Big\vert >up\Vert A\Vert_{F}^2\Big\}\nonumber\\
    \le& 2\exp\Big\{-c(\alpha)\min\{\frac{u^2p\Vert A\Vert_{F}^2}{\Vert  A\Vert_{l_2\to l_2}^2},\big(\frac{up\Vert A\Vert_{F}^2}{\Vert A\Vert_{l_2\to l_2}^2}\big)^{\alpha/2} \} \Big\}.
\end{align}

Let $u=max(\delta, \delta^{2}), \delta\ge 0$. Then, the event 
\begin{align}
    \Big\{\Big\vert\Vert A\xi\Vert_{2}^{2}-p\Vert A\Vert_{F}^2 \Big\vert \le up\Vert A\Vert_{F}^2\Big\}
\end{align}
implies the event
\begin{align}
    \Big\{\Big\vert\Vert A\xi\Vert_{2}-\sqrt{p}\Vert A\Vert_{F} \Big\vert \le \delta\sqrt{p}\Vert A\Vert_{F}\Big\},
\end{align}
due to the numeric bound $\max(\vert z-1\vert, \vert z-1\vert^{2})\le \vert z^2-1\vert, z\ge 0$. Note that $\delta^2=\min(u, u^2)$, then we have by \eqref{Eq_application2}
\begin{align}
    &\mathbb{P}\Big\{\Big\vert\Vert A\xi\Vert_{2}-\sqrt{p}\Vert A\Vert_{F} \Big\vert > \delta\sqrt{p}\Vert A\Vert_{F}\Big\}\nonumber\\
    \le& 2\exp\Big\{-c(\alpha)\min\{\frac{\delta^2p\Vert A\Vert_{F}^2}{\Vert  A\Vert_{l_2\to l_2}^2},\big(\frac{\delta^2p\Vert A\Vert_{F}^2}{\Vert A\Vert_{l_2\to l_2}^2}\big)^{\alpha/2} \} \Big\}.
\end{align}
We finish the proof by taking $t=\delta\sqrt{p}\Vert A\Vert_{F}$.

\end{proof}

\section{Proof of Corollary \ref{Cor_Section2_quadratic_forms}}\label{appendix_B}
\begin{proof}
We first consider the case $0<\alpha\le 1$. Note that 
\begin{align}
  \Vert A\Vert_{\infty} \le \Vert A\Vert_{l_{2}\to l_{\infty}} \le \Vert A\Vert_{l_{2}\to l_{2}}.\nonumber
\end{align}
Therefore, in this case, Corollary \ref{Cor_Section2_quadratic_forms} can be immediately deduced from Lemma \ref{Lem_Lp_quadratic_0_1}.

We next turn to the case $1\le\alpha\le 2$. In this case, note that 
	\begin{align}
		\Vert A\Vert_{l_{\alpha}\to l_{\alpha^{*}}}\le \Vert A\Vert_{l_{2}\to l_{\alpha^{*}}}\le \Vert A\Vert_{l_{2}\to l_{2}}.\nonumber
	\end{align}
Hence, to finish the proof, it is enough to prove for $1\le \alpha\le 2$
\begin{align}
	r^{1/\alpha}\Vert A\Vert_{l_{\alpha^{*}}(l_{2})}\le r^{1/2}\Vert A\Vert_{F}+r\Vert A\Vert_{l_{2}\to l_{\infty}}.\nonumber
\end{align}
Define the following set for $1\le \alpha\le 2$
\begin{align}
	I(r):=\{(x_{ij})=(z_{i}y_{ij})\in \mathbb{R}^{n\times n}: \sum_{i=1}^{n}\vert z_{i}\vert^{\alpha}\le r, \max_{i=1,\cdots, n}\sum_{j=1}^{n}y_{ij}^{2}\le 1    \}.\nonumber
\end{align}
We have the following relation
\begin{align}
	r^{1/\alpha}\Vert A\Vert_{l_{\alpha^{*}}(l_{2})}=\sup\big\{\sum_{i,j}a_{ij}x_{ij}:  (x_{ij})\in I(r)  \big \}.\nonumber
\end{align}
Indeed, on the one hand, we have
\begin{align}
	\sup_{(x_{ij})\in I(r)}\sum_{i,j} a_{ij}z_{i}y_{ij}\le \sup_{(z_{i})}\sum_{i}z_{i}\sup_{(y_{ij})}\sum_{j}a_{ij}y_{ij}\le r^{1/\alpha}\Vert A\Vert_{l_{\alpha^{*}}(l_{2})}.\nonumber
\end{align}
 On the other hand, letting 
 \begin{align}
 	y_{ij}=\frac{a_{ij}}{(\sum_{j=1}^{n}a_{ij}^{2})^{1/2}}, \quad z_{i}=\frac{r^{1/\alpha}(\sum_{j=1}^{n}a_{ij}^{2})^{(\alpha^{*}-1)/2}}{(\sum_{i=1}^{n}(\sum_{j=1}^{n}a_{ij}^{2})^{\alpha^{*}/2})^{1/\alpha}}\nonumber
 \end{align}
yields the inverse inequality. Define another subset of $\mathbb{R}^{n\times n}$
\begin{align}
	I_{1}(r):=\{(x_{ij})=(z_{i}y_{ij})\in \mathbb{R}^{n\times n}: \sum_{i=1}^{n}\min\{\vert z_{i}\vert^{\alpha}, z_{i}^{2}\}\le r, \max_{i=1,\cdots, n}\sum_{j=1}^{n}y_{ij}^{2}\le 1    \}.\nonumber
\end{align}
Obviously, $I(r)\subset I_{1}(r)$. Given any $(z_{i})$ and $(y_{ij})$ satisfying the conditions of $I_{1}(r)$, we have
\begin{align}
	\big\vert \sum_{i,j}a_{ij}z_{i}y_{ij}   \big\vert\le&\sum_{i=1}^{n}\vert z_{i}\vert\big( \sum_{j=1}^{n}a_{ij}^{2}  \big)^{1/2}\big( \sum_{j=1}^{n}y_{ij}^{2}  \big)^{1/2}\le \sum_{i=1}^{n}\vert z_{i}\vert\big( \sum_{j=1}^{n}a_{ij}^{2}  \big)^{1/2}\nonumber\\
	\le & \sum_{i=1}^{n}\vert z_{i}\vert\mathbb{I}_{\{\vert z_{i}\vert \le 1\}}\big( \sum_{j=1}^{n}a_{ij}^{2}  \big)^{1/2}+\sum_{i=1}^{n}\vert z_{i}\vert\mathbb{I}_{\{\vert z_{i}\vert> 1\}}\big( \sum_{j=1}^{n}a_{ij}^{2}  \big)^{1/2}\nonumber\\
	\le &\Vert A\Vert_{F} \big(\sum_{i=1}^{n} z_{i}^{2}\mathbb{I}_{\{\vert z_{i}\vert \le 1\}}\big)^{1/2}+\Vert A\Vert_{l_{2}\to l_{\infty}}\sum_{i=1}^{n}\vert z_{i}\vert\mathbb{I}_{\{\vert z_{i}\vert> 1\}}\nonumber\\
	\le &r^{1/2}\Vert A\Vert_{F}+r\Vert A\Vert_{l_{2}\to l_{\infty}}.\nonumber
\end{align}
This completes the proof. 
\end{proof}

\section{Proof of Lemma \ref{Prop_Section3_Bernstein_sparse}}\label{appendix_C}

\begin{proof}
	Let $\eta_{1}, \cdots, \eta_{n}\stackrel{i.i.d.}{\sim}\mathcal{W}_{s}(\alpha)$, $0<\alpha\le 1$. Note that, for $1\le i\le n$,
$$		-\log \mathbb{P}\{\vert\delta_{i}\eta_{i}\vert\ge t \} = 
		\begin{cases}
			0, & t = 0 \\
			t^{\alpha}-\log p_{i}, & t > 0
		\end{cases}
$$
is a concave function for $t\ge 0$. Hence, Lemma \ref{Lem_Lp_linearsum} yields for $r\ge 2$
\begin{align}\label{Eq_Section3_linearsum_proof1}
	\big\Vert \sum_{i}a_{i}\delta_{i}\eta_{i}\big\Vert_{L_{r}}\asymp_{\alpha} r^{1/\alpha}\big(\sum_{i}p_{i}\vert a_{i}\vert^{r} \big)^{1/r}+\sqrt{r}(\sum_{i}p_{i}a_{i}^{2})^{1/2}.
\end{align}
Here, we use the fact $\Vert \eta_{1}\Vert_{L_{r}}\asymp r^{1/\alpha}$. As for the first term on the right side of \eqref{Eq_Section3_linearsum_proof1}, we have for $r\ge 3$
\begin{align}
	\big(\sum_{i}p_{i}\vert a_{i}\vert^{r} \big)^{1/r}&\le \max_{i}\vert a_{i}\vert^{(r-2)/r}\cdot\big(\sum_{i}p_{i} a_{i}^{2} \big)^{1/r}\nonumber\\
	&= 	\big(e^{2r/(r-2)}\max_{i}\vert a_{i}\vert\big)^{(r-2)/r}\cdot \big(\sum_{i}p_{i} a_{i}^{2}e^{-2r} \big)^{1/r}\nonumber\\
	&\le \frac{r-2}{r}\big(e^{2r/(r-2)}\max_{i}\vert a_{i}\vert\big)+\frac{2}{r}\big(\sum_{i}p_{i} a_{i}^{2}e^{-2r} \big)^{1/2}\nonumber\\
	&\le e^{6}\Vert a\Vert_{\infty} +e^{-r}\big(\sum_{i}p_{i} a_{i}^{2} \big)^{1/2}.\nonumber
\end{align}
This result with \eqref{Eq_Section3_linearsum_proof1} yields for $r\ge 3$,
\begin{align}
	\big\Vert \sum_{i}a_{i}\delta_{i}\eta_{i}\big\Vert_{L_{r}}\lesssim_{\alpha}\sqrt{r}(\sum_{i}p_{i}a_{i}^{2})^{1/2}+r^{1/\alpha}\Vert a\Vert_{\infty}.\nonumber
\end{align}

For $t\ge 0$ and $1\le i\le n$, we have 
\begin{align}
	\mathbb{P}\{ \vert \delta_{i}\zeta_{i}\vert\ge L(\alpha)t  \}\le 2	\mathbb{P}\{ \vert \tilde{\delta}_{i}\eta_{i}\vert\ge t  \},\nonumber
\end{align}
where $\tilde{\delta}_{i}$ is an independent copy of $\delta_{i}$, $1\le i\le n$.
Hence, Lemma \ref{Lem_contraction_principle} yields for $r\ge 3$
\begin{align}
	\big\Vert \sum_{i}a_{i}\delta_{i}\zeta_{i}/L(\alpha)\big\Vert_{L_{r}}\lesssim \big\Vert \sum_{i}a_{i}\tilde{\delta}_{i}\eta_{i}\big\Vert_{L_{r}}\lesssim_{\alpha}\sqrt{r}(\sum_{i}p_{i}a_{i}^{2})^{1/2}+r^{1/\alpha}\Vert a\Vert_{\infty}.\nonumber
\end{align}
Then, using Lemma \ref{Lem_Moments_to_tailbound} and adjusting the universal constant yield the desired result.
\end{proof}

\section{Supplement to Theorem \ref{thm:RIP}}\label{appendix_e}
In this section, we shall give a detailed expansion of $\mathbb{E}\Vert B^\T\mathrm{Diag}(\delta^{(1)})A_{\theta, p}\mathrm{Diag}(\delta^{(1)})B\Vert_{F}^2$ appearing in Theorem \ref{thm:RIP}. For convenience, assume that $B=(b_{ij})_{d\times m}$ and $A_{\theta, p}=(a_{ij})_{d\times d}$.

Note that
\begin{align}
    \Vert B^\T\mathrm{Diag}(\delta^{(1)})A_{\theta, p}\mathrm{Diag}(\delta^{(1)})B\Vert_{F}^2=&\sum_{i=1}^{m}\sum_{j=1}^{m}\big(\sum_{l=1}^{d}\sum_{k=1}^{d}\delta_lb_{li}a_{lk}\delta_{k}b_{kj} \big)^{2}\nonumber\\
    =&\sum_{i=1}^{m}\sum_{j=1}^{m}\sum_{l=1}^{d}\sum_{p=1}^{d}\sum_{k=1}^{d}\sum_{q=1}^{d}\delta_{l}b_{li}a_{lk}\delta_{k}b_{kj}\delta_{p}b_{pj}a_{pq}\delta_{q}b_{qj}.
\end{align}
In order to calculate the expectation of the above random variables, we need to make the following classifications based on whether the indicators $l, p, k$, and $q$ are equal.
\begin{itemize}
    \item All indicators are different.
    \begin{align}
        \mathrm{I}_{1}=\{(l, p, k, q):1\le l\neq p\neq k\neq q\le d \}.
    \end{align}
    \item Only two indicators are equal.
    \begin{align}
        \mathrm{I}_{2}&=\{(l, p, k, q):1\le l= p\neq k\neq q\le d \},\quad
        \mathrm{I}_{3}=\{(l, p, k, q):1\le l\neq p= k\neq q\le d \},\\
        \mathrm{I}_{4}&=\{(l, p, k, q):1\le l\neq p\neq k= q\le d \},\quad
        \mathrm{I}_{5}=\{(l, p, k, q):1\le l=k\neq p\neq q\le d \},\\
        \mathrm{I}_{6}&=\{(l, p, k, q):1\le l=q\neq p\neq k\le d \},\quad
        \mathrm{I}_{7}=\{(l, p, k, q):1\le l\neq p=q\neq k\le d \}.
    \end{align}
    \item Only three indicators are equal.
    \begin{align}
        \mathrm{I}_{8}&=\{(l, p, k, q):1\le l= p= k\neq q\le d \},\quad
        \mathrm{I}_{9}=\{(l, p, k, q):1\le q= p= k\neq l\le d \},\\
        \mathrm{I}_{10}&=\{(l, p, k, q):1\le l= q= k\neq p\le d \},\quad
        \mathrm{I}_{11}=\{(l, p, k, q):1\le l= p= q\neq k\le d \}.
    \end{align}
    \item All indicators are equal.
    \begin{align}
        \mathrm{I}_{12}=\{(l, p, k, q):1\le l= p= q= k\le d \}.
    \end{align}
\end{itemize}
Hence, we have 
\begin{align}
    \mathbb{E}\Vert B^\T\mathrm{Diag}(\delta^{(1)})A_{\theta, p}\mathrm{Diag}(\delta^{(1)})B\Vert_{F}^2=\sum_{i=1}^{m}\sum_{j=1}^{m}\sum_{w=1}^{12}Q_{w}(i,j).
\end{align}
Here, 
\begin{align}
    Q_{1}(i,j)&=\sum_{(l, k, p, q)\in I_1}b_{li}\theta_{l}\theta_{k}b_{kj}b_{pj}\theta_p\theta_qb_{qj},\quad Q_{2}(i,j)=\sum_{(l, k, p, q)\in I_2}\frac{1}{p_l}b_{li}\theta_{l}\theta_{k}b_{kj}b_{pj}\theta_p\theta_qb_{qj},\\
    Q_{3}(i,j)&=\sum_{(l, k, p, q)\in I_3}\frac{1}{p_k}b_{li}\theta_{l}\theta_{k}b_{kj}b_{pj}\theta_p\theta_qb_{qj},\quad Q_{4}(i,j)=\sum_{(l, k, p, q)\in I_4}\frac{1}{p_k}b_{li}\theta_{l}\theta_{k}b_{kj}b_{pj}\theta_p\theta_qb_{qj},\\
    Q_{5}(i,j)&=\sum_{(l, k, p, q)\in I_5}b_{li}\theta_{l}\theta_{k}b_{kj}b_{pj}\theta_p\theta_qb_{qj},\quad Q_{6}(i,j)=\sum_{(l, k, p, q)\in I_6}\frac{1}{p_l}b_{li}\theta_{l}\theta_{k}b_{kj}b_{pj}\theta_p\theta_qb_{qj},\\
    Q_{7}(i,j)&=\sum_{(l, k, p, q)\in I_7}b_{li}\theta_{l}\theta_{k}b_{kj}b_{pj}\theta_p\theta_qb_{qj},\quad Q_{8}(i,j)=\sum_{(l, k, p, q)\in I_8}\frac{1}{p_l}b_{li}\theta_{l}\theta_{k}b_{kj}b_{pj}\theta_p\theta_qb_{qj},\\
    Q_{9}(i,j)&=\sum_{(l, k, p, q)\in I_9}\frac{1}{p_k}b_{li}\theta_{l}\theta_{k}b_{kj}b_{pj}\theta_p\theta_qb_{qj},\quad Q_{10}(i,j)=\sum_{(l, k, p, q)\in I_{10}}\frac{1}{p_k}b_{li}\theta_{l}\theta_{k}b_{kj}b_{pj}\theta_p\theta_qb_{qj},\\
    Q_{11}(i,j)&=\sum_{(l, k, p, q)\in I_{11}}\frac{1}{p_l}b_{li}\theta_{l}\theta_{k}b_{kj}b_{pj}\theta_p\theta_qb_{qj},\quad Q_{12}(i,j)=\sum_{(l, k, p, q)\in I_1}\frac{1}{p_l}b_{li}\theta_{l}\theta_{k}b_{kj}b_{pj}\theta_p\theta_qb_{qj}.
\end{align}

\end{document}